\documentclass[10pt]{extarticle}
\usepackage{amsmath}
\usepackage[dvips]{graphicx}
\usepackage{textcomp}
\usepackage{amsbsy}
\usepackage{latexsym}
\usepackage[mathscr]{eucal}
\usepackage{amsfonts}
\usepackage{amssymb}
\usepackage[usenames]{color}
\usepackage{stmaryrd}
\usepackage{amsthm}
\usepackage{mathabx}

\usepackage[utf8]{inputenc}

\usepackage[T1]{fontenc}
\usepackage{lmodern}

\usepackage{algorithm,algorithmic}

\usepackage{epsf}

\usepackage{fullpage}

\usepackage{stackrel}

\usepackage{hyperref}
\hypersetup{ hidelinks = true, } 

\usepackage{comment}

\def\RR{\rm \hbox{I\kern-.2em\hbox{R}}}
\def\NN{\rm \hbox{I\kern-.2em\hbox{N}}}
\def\ZZ{\rm {{\rm Z}\kern-.28em{\rm Z}}}
\def\CC{\rm \hbox{C\kern -.5em {\raise .32ex \hbox{$\scriptscriptstyle
|$}}\kern
-.22em{\raise .6ex \hbox{$\scriptscriptstyle |$}}\kern .4em}}

\def\<{\langle}
\def\>{\rangle}

\def\e{\varepsilon}

\def\cF{{\cal F}}

\newcommand{\R}{{\mathbb R}}

\def\Chi{\raise .3ex
\hbox{\large $\chi$}} 
\def\lsima{\hbox{\kern -.6em\raisebox{-1ex}{$~\stackrel{\textstyle<}{\sim}~$}}\kern -.4em}
\def\lsim{\hbox{\kern -.2em\raisebox{-1ex}{$~\stackrel{\textstyle<}{\sim}~$}}\kern -.2em}
\def\[{\Bigl [}
\def\]{\Bigr ]}
\def\({\Bigl (}
\def\){\Bigr )}
\def\[{\Bigl [}
\def\]{\Bigr ]}
\def\({\Bigl (}
\def\){\Bigr )}

\newcommand{\be}{\begin{equation}}
\newcommand{\ee}{\end{equation}}
\newcommand{\beu}{\begin{equation*}}
\newcommand{\eeu}{\end{equation*}}
\newcommand{\bea}{$$ \begin{array}{lll}}
\newcommand{\eea}{\end{array} $$}
\newcommand{\bi}{\begin{itemize}}
\newcommand{\ei}{\end{itemize}}

\newtheorem{theorem}{Theorem}
\newtheorem{remark}{Remark}
\newtheorem{lemma}{Lemma}

\newtheorem{corollary}{Corollary}

\DeclareMathOperator*{\argmin}{argmin}

\newcommand{\newmes}{\chi}
\newcommand{\errmes}{\rho}
\newcommand{\mymes}{\gamma}
\newcommand{\auxmes}{\mu}

\newcommand{\Gaus}{U}
\newcommand{\binrv}{B}
\newcommand{\Bin}{\textrm{Bin}}

\newcommand{\sample}{x}
\newcommand{\spazio}{n}
\newcommand{\iterazmax }{ t }
\newcommand{\iteraz}{k}
\newcommand{\indbasis}{i}
\newcommand{\indsample}{j}
\newcommand{\punti}{m}

\newcommand{\param}{r}
\newcommand{\paramzeta}{s}

\newcommand{\conf}{\alpha}
\newcommand{\basis}{\psi}

\newcommand{\intero}{\tau}
\newcommand{\costprop}{\theta}
\newcommand{\thresholdcond}{\xi}
\newcommand{\paradorfler}{\beta}
\newcommand{\setx}{X}
\newcommand{\bor}{\mathfrak{B}}

\begin{document}
\bibliographystyle{plain}
\title
{
Adaptive approximation by optimal weighted least squares methods
}
\author{
Giovanni Migliorati\thanks{
Sorbonne Universit\'e, UPMC Univ Paris 06, CNRS, UMR 7598, Laboratoire Jacques-Louis Lions, 4, place Jussieu 75005, Paris, France. 
 email: migliorati@ljll.math.upmc.fr
} }


\date{\today}

\maketitle

\begin{abstract}
\noindent 
Given any domain $\setx\subseteq \R^d$ and a probability measure $\errmes$ on $\setx$, we study the problem of approximating in $L^2(\setx,\errmes)$ a given function $u:\setx\to\R$, using its noiseless pointwise evaluations at random samples. For any given linear space $V\subset L^2(\setx,\errmes)$ with dimension $n$, previous works have shown that stable and optimally converging Weighted Least-Squares (WLS) estimators can be constructed using $m$ random samples distributed according to an auxiliary probability measure $\auxmes$ that depends on $V$, with $m$ being linearly proportional to $n$ up to a logarithmic term. As a first contribution, we present novel results on the stability and accuracy of WLS estimators with a given approximation space, using random samples that are more structured than those used in the previous analysis. As a second contribution, we study approximation by WLS estimators in the adaptive setting. For any sequence of nested spaces $(V_\iteraz)_{\iteraz} \subset L^2(\setx,\errmes)$, we show that a sequence of WLS estimators of $u$, one for each space $V_\iteraz$, can be sequentially constructed such that: i) the estimators remain provably stable with high probability and optimally converging in expectation, simultaneously for all iterations from one to $\iteraz$, and ii) the overall number of samples necessary to construct all the first $\iteraz$ estimators remains linearly proportional to the dimension of $V_\iteraz$, up to a logarithmic term. The overall number of samples takes into account all the samples generated to build all the estimators from iteration one to $\iteraz$. We propose two sampling algorithms that achieve this goal. The first one is a purely random algorithm that recycles most of the samples from the previous iterations. The second algorithm recycles all the samples from all the previous iterations. Such an achievement is made possible by crucially exploiting the structure of the random samples. Finally we apply the results from our analysis to develop numerical methods for the adaptive approximation of functions in high dimension. 
\end{abstract}


\section{Introduction}
\noindent
In recent years, the increasing computing power and availability of data have contributed to a huge growth in the complexity of the mathematical models.  
Dealing with such models often requires the approximation or integration of functions in high dimension, that can be a challenging task due to the curse of dimensionality.  
The present paper studies the problem of approximating a function $u:\setx\to\mathbb{R}$ that depends on a $d$-dimensional parameter $\sample \in \setx \subseteq \mathbb{R}^d$, 
using the information coming from the evaluations of $u$ at a set of selected samples $\sample^1,\ldots,\sample^\punti\in \setx$.  
Two classical  approaches to such a problem are \emph{interpolation} and \emph{least-squares methods}, see \emph{e.g.}~\cite{D,BG2004a,GKKW}.  
Here we turn our attention to least-squares methods, that are frequently used in applications for approximation, data-fitting, estimation and prediction. 
Other approaches to function approximation are \emph{compressive sensing}, see \cite{FR} and references therein, and \emph{neural networks}, see \emph{e.g.}~\cite{C,L}.

Previous convergence results for standard least-squares methods have been proposed in \cite{CDL}, in expectation, and \cite{MNST2014}, in probability. 
Weighted Least-Squares methods (hereafter WLS) have been previously studied in \cite{DH,JNZ,CM2016}. 
It has been proven in \cite{CM2016} that stable and optimally converging WLS estimators can be constructed using judiciously distributed random samples, whose number is only linearly proportional to the dimension of the approximation space, up to a logarithmic term. 
The analysis holds in general approximation spaces, and the number of samples ensuring stability and optimality of the estimator does not depend on $d$.   
Such a result is recalled in Theorem~\ref{theo2}. 
The analysis in \cite{CM2016} considers both cases of noisy or noiseless evaluations of $u$.
In this paper we confine to the case of noiseless evaluations, which is relevant whenever the function $u$ can be evaluated at the selected samples with sufficiently high precision, \emph{e.g.}~up to machine epsilon. 
The case of noisy evaluations can be addressed using the same techniques as in \cite{CM2016,MNT2015}.    

The proof of Theorem~\ref{theo2}, and more generally the analysis in \cite{CM2016}, use results from \cite{AW,Tropp} on tail bounds for sums of random matrices. An interesting 
feature of the bounds in \cite{Tropp}  
is that the matrices need not be identically distributed. 
The analysis in \cite{CM2016} does not take advantage of this property.  
One of the main goals of the present paper is to show how the use of this property paves the way towards novel results in the analysis of WLS methods for a given space (Theorem~\ref{theo_mio}), and towards their application in an adaptive setting (Theorem~\ref{thm:teo_union_bounds_mio}). 
The proof of Theorem~\ref{theo_mio} builds on previous contributions \cite{CDL,CM2016}. 
The overall skeleton of the proof is similar,  but with some crucial differences that make use of the additional structure of the random samples. 

The outline of the paper is the following: in Section~\ref{sec:intro_main_novelty} we describe and motivate our contributions. In Section~\ref{sec:opt_wei_ls} we recall some results from the analysis in \cite{CM2016} on weighted least squares for a given space. Section~\ref{sec:opt_wei_ls_mio} contains Theorem~\ref{theo_mio} and its proof. In Section~\ref{sec:adaptive} we apply Theorem~\ref{theo_mio} and Theorem~\ref{theo2} to the adaptive setting, with an arbitrary nested sequence $(V_\iteraz)_\iteraz$ of approximation spaces. 
In Section 4 we present the sampling algorithms. Section 5 contains some numerical tests, together with an example of adaptive algorithm that uses sequences of nested polynomial spaces. 
Section 6 draws some conclusions. 
All the algorithms are collected in appendix.

\subsection{Motivations and outline of the main results}
\label{sec:intro_main_novelty}
\noindent
Let $\setx\subseteq \R^d$ be a Borel set, $\errmes$ be a Borel probability measure on $\setx$, $(\basis_\indbasis)_{\indbasis \geq 1}$ be a basis orthonormal in $L^2(\setx,\errmes)$ equipped with the inner product $\langle f_1, f_2  \rangle = \int_\setx f_1(\sample)\,  f_2(\sample) \, d\errmes$, and $V:=\textrm{span}\{ \basis_1,\ldots,\basis_\spazio \}$ be the space obtained by retaining $\spazio$ terms of the basis.
The least-squares method approximates the function $u$ by computing its discrete $L^2$ projection onto a given space $V$, using pointwise evaluations of $u$ at a set of $\punti\geq \spazio$ distinct random samples $\sample^1,\ldots,\sample^\punti$.  
The analysis in \cite{CM2016}, whose main findings are resumed in the forthcoming Theorem~\ref{theo2}, provides some results on the stability and convergence properties of such a discrete projection, and of other WLS estimators as well.  
In Theorem~\ref{theo2}, independent and identically distributed random samples are drawn from the probability measure
$$
d\auxmes_\spazio = \dfrac{1}{\spazio} \sum_{j=1}^\spazio d\newmes_j,  
$$
that is an additive mixture of the probability measures $\newmes_j$ defined as 
\begin{equation}
\newmes_j(A):=\int_A |\basis_j(\sample)|^2 \, d\errmes, \qquad \textrm{for any Borel set $A\subseteq \setx$.}
\label{eq:def_mix_mes}
\end{equation}
One sample from $\auxmes_\spazio$ can be generated by randomly choosing an index $j$ uniformly in $\{1,\ldots,\spazio\}$ and then drawing one sample from $\newmes_j$. 
In general $\auxmes_\spazio$ is not a product measure, even if $\errmes$ is a product measure. 

Another novel approach proposed in this paper uses a different type of random samples. Such an approach uses a set of independent random samples of the form $\sample^1, \ldots, \sample^{\punti}$ with $\punti=\intero \spazio$ for a suitable integer $\intero$, and 
such that for any $j=1,\ldots,\spazio$, the samples $\sample^{(j-1)\intero +1}, \ldots, \sample^{j \intero}$ are distributed according to $\newmes_j$. 
These samples are not identically distributed.
On the upside, they are more structured than those used in Theorem~\ref{theo2}, since the amount of samples coming from each component of the mixture is fixed.  
If $\intero=1$ then the $\spazio$ independent samples $\sample^1,\ldots,\sample^\spazio$ are jointly drawn from 
$$
(\sample^1,\ldots,\sample^\spazio) \sim d\mymes_\spazio := 
\prod_{j=1}^\spazio | \basis_j(\sample^j) |^2 \, d\errmes. 
$$
If $\intero \geq 1$ the draw of $\punti=\intero \spazio$ independent samples $\sample^1,\ldots,\sample^\punti$ follows the probability measure $d\mymes^{\punti} := \otimes^\intero d\mymes_\spazio$.

Denote with $d\auxmes^\punti := \otimes^\punti d\auxmes_\spazio$ the probability measure for the draw of $\punti$ i.i.d. samples from $\auxmes_\spazio$. 
Given a fixed $\spazio$, in the limit $\punti=\intero \spazio\to\infty$ obtained by $\intero \to \infty$, the proportion of random samples of $\auxmes^\punti$ coming from each $\newmes_j$ 
tends to $1/\spazio$ by the strong law of large numbers, whereas the same proportion is exactly equal to $1/\spazio$ by construction for the samples drawn from $\mymes^\punti$.  
With any $\punti$ the two probability measures $\auxmes^\punti$ and $\mymes^\punti$ generate samples with different distributions. 
However, the block of $\spazio$ samples from $\mymes_\spazio$ still mimics 
the samples from $\auxmes_\spazio$. 
For example the sum of the expectation of the random samples is preserved,  
$$
(\sample^1,\ldots,\sample^\spazio) \sim d\mymes_\spazio,
\quad
\tilde \sample \sim d\auxmes_\spazio 
\implies 
 \sum_{j=1}^\spazio \mathbb{E}\left(\sample^j \right)  
=
\sum_{j=1}^\spazio \int_\setx  \sample^j  |\basis_j( \sample^j)|^2 \, d\errmes  
=
\int_\setx \tilde \sample \sum_{j=1}^\spazio |\basis_j(\tilde \sample)|^2 \, d\errmes  
=
\spazio  \mathbb{E}(\tilde \sample),  
$$
and this preservation plays a main role in our forthcoming analysis. 
All the measures appearing in this paper are also Borel probability measures, and sometimes for brevity we refer to them just as measures.

The first main result of this paper is Theorem~\ref{theo_mio}. 
It proves the same guarantees as  Theorem~\ref{theo2} 
for the stability and accuracy of WLS estimators 
with a given approximation space, but when 
the random samples from $\auxmes^\punti$ 
are replaced with random samples from $\mymes^\punti$.
The second main result concerns the analysis of WLS estimators, when considering a sequence of nested approximation spaces $(V_\iteraz)_{\iteraz\geq 1}$, 
where $V_\iteraz:=\textrm{span}\{\basis_1,\ldots,\basis_{\spazio_\iteraz}\}$ and 
$\spazio_\iteraz:=\textrm{dim}(V_\iteraz)$. 
In this adaptive setting, we compare the two approaches using random samples from  $\mymes^\punti$ or $\auxmes^\punti$.   
In both cases, in Theorems~\ref{thm:teo_union_bounds_mio} and \ref{thm:teo_union_bounds}, we prove that a sequence of estimators of $u$, one for each space $V_\iteraz$, can be sequentially constructed 
such that: i) the estimators remain provably stable with high probability and optimally converging in expectation, simultaneously for all iterations from one to $\iteraz$, and ii) the overall number of samples necessary to construct all the first $\iteraz$ estimators remains linearly proportional to the dimension of $V_\iteraz$, up to a logarithmic term. As a further contribution we show that using the samples from $\mymes^\punti$ rather than from $\auxmes^\punti$ provides the following advantages, that are relevant in the development of adaptive WLS methods. 
\begin{itemize}

\item {\bf Structure of the random samples.} When using $\mymes^\punti$, the number of random samples coming from each component $|\basis_j(\sample)|^2\, d\errmes$ of the mixture is precisely determined, and allows the development of adaptive algorithms that recycle all the samples from all the previous iterations. When using $\auxmes^\punti$ it is not possible to recycle all the samples from the previous iterations with probability equal to one. Given two nested spaces $V_{\iteraz-1}\subset V_{\iteraz}$ and two positive integers $\intero_{\iteraz-1}\leq \intero_{\iteraz}$,   
at iteration $\iteraz$ the probability measure $\mymes^{\punti_\iteraz}$ of the $\punti_{\iteraz}= \intero_\iteraz \spazio_\iteraz$ samples can be decomposed as 
\begin{align}
\label{eq:decomposition_mes_chi}
d\mymes^{\punti_\iteraz}
= &
\otimes^{\intero_\iteraz}
\prod_{j=1}^{\spazio_{\iteraz }} 
 |\basis_j(\sample)|^2 \, d\errmes 
\\
= &
\left( 
\underbrace{
d\mymes^{\punti_{\iteraz-1}}
}_{
\substack{
\textrm{measure of the $\punti_{\iteraz-1}$ }\\
\textrm{samples recycled with }\\   
\textrm{ certainty from step } \iteraz-1
}
}
\right)
\otimes
\left( 
\underbrace{
\otimes^{\intero_\iteraz-\intero_{\iteraz-1}}
\prod_{j=1}^{\spazio_{\iteraz-1}} 
 |\basis_j(\sample)|^2 \, d\errmes 
}_{
\textrm{measure of the new samples drawn}   
\atop
\textrm{from the old components of the mixture}
}
\right)
\otimes
\left(
\underbrace{
 \otimes^{\intero_\iteraz} 
\prod_{j=1+\spazio_{\iteraz-1}}^{\spazio_{\iteraz }}
|\basis_{j }(\sample)|^2 \, d\errmes  
}_{\textrm{measure of the new samples drawn}
\atop  
\textrm{from the new components of the mixture}
}
\right). 
\nonumber
\end{align}
The probability measure $d\auxmes^{\punti_\iteraz} =\otimes^{\punti_\iteraz} d\auxmes_{\spazio_\iteraz}$ cannot be decomposed as the product of two probability measures with one being $\auxmes^{\punti_{\iteraz-1}}$,  
because $\auxmes_{\spazio_\iteraz}$ is not a product measure. It is however possible to leverage the structure of $\auxmes_{\spazio_\iteraz}$ as an additive mixture of $\auxmes_{\spazio_{\iteraz-1}}$ and a suitable probability measure $\sigma_{\spazio_{\iteraz}}$ , and decompose $\auxmes^{\punti_\iteraz}$ as 
\begin{align}
d\auxmes^{\punti_\iteraz}
 = & 
\otimes^{\punti_\iteraz} \left( 
\dfrac{\spazio_{ \iteraz -1 }}{\spazio_{\iteraz}}
\underbrace{
\overbrace{
\dfrac{1}{\spazio_{\iteraz-1}}  
\sum_{j=1}^{\spazio_{\iteraz-1}} | \basis_j(\sample)|^2 \, d\errmes   
}^{d\auxmes_{\spazio_{\iteraz-1}}}
}_{
\substack{
\textrm{measure of the samples drawn} \\ 
\textrm{from the old components of $d\auxmes_{\spazio_{\iteraz}}$,} \\
\textrm{perhaps recycled from step $\iteraz-1$}
}
}
+
\dfrac{\spazio_{\iteraz} - \spazio_{\iteraz-1}}{\spazio_{\iteraz}}
\underbrace{
\overbrace{
\dfrac{1}{\spazio_{\iteraz} - \spazio_{\iteraz-1}} 
\sum_{j=1+\spazio_{\iteraz-1}}^{\spazio_\iteraz} | \basis_j(\sample)|^2 \, d\errmes    
}^{d\sigma_{\spazio_{\iteraz}}}
}_{
\substack{
\textrm{measure of the samples drawn}\\
\textrm{from the new components of $d\auxmes_{\spazio_{\iteraz}}$}\\
}
}
\right). 
\label{eq:decomposition_mixture}
\end{align}
When drawing $\punti_\iteraz$ samples from $\auxmes_{\spazio_{\iteraz}}$, the amount of samples coming from one of the components of $\auxmes_{\spazio_{\iteraz-1}}$ is a binomial random variable with number of trials $\punti_\iteraz$ and rate of success $\spazio_{\iteraz-1}/\spazio_{\iteraz}$ for each trial. 
Whenever this random variable takes values less than $\punti_{\iteraz-1}$, that always occurs with some positive probability, it is not possible to recycle all the $\punti_{\iteraz-1}$ samples from iteration $\iteraz-1$. 
 
\item {\bf Variance reduction.} 
Random mixture proportions induce extra variance in the generated samples. 
As a consequence,  random samples from $\mymes^\punti$ are more disciplined than random samples from $\auxmes^\punti$. 
This stabilization effect amplifies when using basis elements whose supports are more localized than globally supported orthogonal polynomials.
More on this in Remark~\ref{thm:stabil_localiz}. 

\item {\bf Coarsening and extension to nonnested sequences of approximation spaces.} 
When using the samples from $\mymes^\punti$, thanks to the decomposition \eqref{eq:decomposition_mes_chi}, it is possible to remove an element of the basis $\basis_j$ from the space $V$ as well as its associated samples  $\sample^{(j-1)\intero +1}, \ldots, \sample^{j \intero}$ from the whole set $\sample^1, \ldots, \sample^{\punti}$ of $\punti=\intero \spazio$ samples, 
and at the same time recycle all the $\intero (\spazio-1)$ remaining samples for $V \setminus \{ \basis_j \}$. 
More generally, the use of $\mymes^\punti$ allows the development of efficient adaptive methods with arbitrary sequences of approximation spaces $(V_\iteraz)_\iteraz$, that probe any $\basis_j\notin V_\iteraz$ 
chosen according to some criterion. The method then either retains $\basis_j$ as $V_{\iteraz+1}=V_\iteraz \cup \{ \basis_j\}$ or discards it 
depending on its contribution to the reduction of (an estimator of) the error from $V_\iteraz$ to $V_{\iteraz+1}$.

\end{itemize}

\paragraph{Comparison with \cite{BC2018}.}
Section \ref{sec:seq_random_sampling} presents an analysis of a sampling algorithm (Algorithm~\ref{algo:sampling_altro}) that sequentially generates $\punti$ random samples from $\auxmes^\punti$ 
with an arbitrary nested sequence of approximation spaces $(V_\iteraz)_\iteraz$.
In \cite{BC2018} a similar algorithm that uses $\auxmes^\punti$ has been proposed and analysed. 

\newcommand{\mesnot}{\errmes}
\paragraph{Notation for product of measures.}
Let $\mesnot_1,\mesnot_2$ be two Borel measures on $\setx\subseteq \mathbb{R}^d$ with the Borel $\sigma$-algebra 
$\bor=\bor(\setx)$.  
The notation $\mesnot_1 \otimes \mesnot_2$ denotes the product measure on $\setx\times \setx$ with the tensor product Borel $\sigma$-algebra $\bor\otimes\bor$,  that satisfies
$$
\mesnot_1 \otimes \mesnot_2 (A_1\times A_2) = \mesnot_1(A_1) \mesnot_2(A_2), \quad \textrm{ for any } A_1, A_2 \in \bor.
$$

\section{Optimal weighted least squares for a given approximation space}
\label{sec:opt_wei_ls}
\subsection{Previous results}
\label{sec:opt_wei_ls_prev}
\noindent
Let $\setx\subseteq \R^d$ be a Borel set, and $\errmes$ be a Borel probability measure on $\setx$.  
We define the $L^2(\setx,\errmes)$ inner product 
\begin{equation}
\langle f_1, f_2  \rangle = \int_\setx  
f_1(\sample)\,  f_2(\sample) \, d\errmes(\sample)
\label{eq:L2innerprod}
\end{equation}
associated with the norm $\| f\|:= \langle f,f\rangle^{1/2}$. 
Throughout the paper we denote by $(\basis_\indbasis)_{\indbasis \geq 1}$ 
an $L^2(\setx,\errmes)$ orthonormal basis. 
We define the linear space $V:=\textrm{span}\{ \basis_1,\ldots,\basis_\spazio \}$ that contains $\spazio$ arbitrarily chosen elements of the basis,  
and denote with $\spazio:=\textrm{dim}(V)$ its dimension. 
We further assume that 
\begin{equation}
\textrm{ for any $\sample \in \setx$ there exists $\basis_j \in V$ such that $\basis_j(\sample)\neq 0$.}
\label{eq:assumption_christoff}
\end{equation}
This assumption is verified for example if the space $V$ contains the functions that are constant over $\setx$. 
For any given $V$, we define the weight function $w:\setx \to \R$ as 
\begin{equation}
w(\sample):=\dfrac{ \spazio }{ \sum_{\indbasis =1}^\spazio | \basis_\indbasis(\sample)  |^2  }, \quad \sample \in \setx, 
\label{eq:weight_function_def} 
\end{equation}
whose denominator 
does not vanish 
under assumption \eqref{eq:assumption_christoff}. 
The function $w$ is known as the Christoffel function, up to a renormalization, when $V$ is the space of algebraic polynomials with prescribed total degree. 
Using $w$ we define the probability measure 
\begin{equation}
d\auxmes_{\spazio}:= w^{-1} d\errmes=  \dfrac{ \sum_{\indbasis=1}^{ \spazio } | \basis_\indbasis(\sample)|^2  }{\spazio}  d\errmes,  
\label{eq:changemes}
\end{equation}
which depends on the chosen approximation space $V$. 
Another inner product used in this paper is 
\begin{equation}
\langle f_1,f_2  \rangle_\punti := \dfrac{1}{\punti} \sum_{\indsample=1}^\punti w(\sample^\indsample)  f_1(\sample^\indsample) f_2(\sample^\indsample), 
\label{eq:disc_inn_prod}
\end{equation}
where the functions $w$, $f_1$, $f_2$ are evaluated at $\punti$ samples $\sample^1,\ldots,\sample^\punti$ independent and identically distributed as $\auxmes_\spazio$.
This inner product is associated with the discrete seminorm $\| f\|_\punti:= \langle f,f \rangle_\punti^{1/2}$.  
The discrete inner product \eqref{eq:disc_inn_prod} mimics \eqref{eq:L2innerprod} due to \eqref{eq:changemes}.  
The exact $L^2$ projection on $V$ of any function $u \in L^2(\setx,\errmes)$ is defined as 
$$
\Pi_\spazio u := \argmin_{v \in V }  \| u - v \|. 
$$

In practice such a projection cannot be computed out of very particular cases, motivating the interest towards the discrete least-squares approach. 
We define the weighted least-squares estimator $u_W$ of $u$ as  
$$
u_W := \Pi_\spazio^\punti u = \argmin_{v \in V }  \| u -v \|_\punti,    
$$
that is obtained by applying the discrete projector $\Pi_\spazio^\punti$ on $V$ to $u$. 
The estimator $u_W$ is associated to 
the 
solution of the linear system 
\begin{equation}
G a = h, 
\label{eq:normal_equations}
\end{equation}
where the Gramian matrix $G$ and the right-hand side $h$ are defined element-wise as 
$$
G_{ij}=\langle \basis_i , \basis_j   \rangle_\punti,
\quad h_i=\langle u, \basis_i \rangle_\punti, 
\quad 
i,j=1,\ldots,\spazio, 
$$
and $a=(a_1,\ldots,a_\spazio)^\top$ is the vector containing the coefficients of $u_W=\sum_{\indbasis=1}^\spazio a_\indbasis \basis_\indbasis$ expanded over the orthonormal basis. 
The linear system \eqref{eq:normal_equations} always has at least one solution, which is unique when $G$ is nonsingular. 
When $G$ is singular we can define $u_W$ as the estimator associated to the unique minimal $\ell_2$-norm solution to \eqref{eq:normal_equations}.
Moreover, we define the $L^2(\setx,\errmes)$ best approximation of $u$ in $V$ as 
\begin{equation}
e_{\spazio,2}(u):= \min_{v \in V} \| u- v  \| =  \| u- \Pi_\spazio u  \|, 
\label{eq:def_best_approx_l2}
\end{equation}
and the weighted $L^\infty(\setx,\errmes)$ best approximation of $u$ as 
$$
e_{\spazio,\infty,w}(u):= \inf_{v \in V}  \sup_{y \in \setx}  \sqrt{w(y)} | u(y) - v(y) |. 
$$
Notice that 
$\Pi_\spazio$, $\Pi_\spazio^\punti$, $e_{\spazio,2}$ and $e_{\spazio,\infty,w}$ 
depend on the chosen space $V$, and not only on its dimension $\spazio$. 
The identity matrix is denoted with $I \in \mathbb{R}^{\spazio\times\spazio}$. 
The spectral norm of any matrix $A\in \mathbb{R}^{\spazio\times\spazio}$ is defined as 
$$
\vvvert A \vvvert:=
\sup_{
\| v \|_{\ell_2} =1
} 
\| Av \|_{\ell_2}, 
$$ 
using the Euclidean inner product in $\mathbb{R}^\spazio$ and its associated norm. 
Another weighted least-squares estimator introduced in \cite{CM2016} is the conditioned estimator:  
\begin{equation}
\label{eq:def_cond_est}
u_C 
:= 
\begin{cases}
u_W, & \textrm{ if } \vvvert  G - I \vvvert \leq \frac12,   \\
0, & \textrm{ otherwise}. 
\end{cases} 
\end{equation}

One of the main results from \cite{CM2016} is the following theorem, 
see \cite[Theorem 2.1 and Corollary 2.2]{CM2016}.   
\begin{theorem}
\label{theo2}
For any real $\param>0$, if the integers $\punti$ and $\spazio$ are such that the condition 
\begin{equation}
\spazio\, \costprop_\param
\leq  \frac \punti {\ln \punti},\;\; \textrm{ with } \; \; 
\costprop_\param:=\costprop^{-1} (1+\param), 
\quad 
\costprop:=\frac{3\ln(3/2)-1}{2} \approx 0.108,
\label{condmw}
\end{equation}
is fulfilled, 
and $\sample^1,\ldots,\sample^\punti$ are independent and identically distributed random samples from $\auxmes_\spazio$, then the following holds:
\begin{enumerate}
\item[{\rm (i)}] the matrix $G$ satisfies the tail bound
\begin{equation*}
{\rm Pr} \, \left\{ \vvvert G- I\vvvert > \frac12 \right\} \leq
2\spazio \punti^{-(\param+1)} 
\leq 
2\punti^{-r}; 
\end{equation*}

\item[{\rm (ii)}] if $u\in L^2(\setx,\errmes)$ then
the estimator $u_C$ 
satisfies  
\begin{equation*}
\mathbb{E}(\|u-u_C\|^2)\leq (1+\e(\punti))e_{\spazio,2}(u)^2+2\|u\|^2 \punti^{-\param},
\end{equation*}
where $\e(\punti):= \frac {4} {\costprop_\param \, \ln (\punti)}\to 0$ as $\punti\to +\infty$, and $\costprop_\param$ as in {\rm \eqref{condmw}};
\item[{\rm (iii)}] with probability larger than $1-2\punti^{-\param}$, the 
estimator $u_W$ satisfies
\begin{equation}
\|u-u_W\|\leq (1+\sqrt 2) e_{\spazio,\infty,w}(u),
\label{nearoptimalprobw_2}
\end{equation}
for all $u$ such that $\| \sqrt{w} u \|_{L^\infty}<+\infty$. 
\end{enumerate}
\end{theorem}

The above theorem can be rewritten for a chosen confidence level, by setting $\conf=2\spazio \punti^{-(\param +1)}$ and replacing the corresponding $\param$ in \eqref{condmw}. 
For convenience we rewrite condition \eqref{condmw} with equality using the ceiling operator, since the number of samples is an integer and usually one wishes to minimize the number of samples $\punti$  satisfying \eqref{condmw} for a given $\spazio$.   

\begin{corollary}
\label{thm:coro_alpha}
For any $\conf \in (0,1)$ and any integer $\spazio\geq 1$,
if  
\begin{equation}
\punti  = \left\lceil \dfrac{ \spazio }{\costprop}  
\ln\left( 
\dfrac{2 \spazio }{ \conf }
\right)
\right\rceil, 
\quad \textrm{with $\costprop$ as in \eqref{condmw}},  
\label{eq:cond_alpha_k}
\end{equation}
and $\sample^1,\ldots,\sample^\punti$ are $\punti$ independent and identically distributed random samples from $\auxmes_\spazio$,
 then 
$$
\textrm{Pr}\left(  \vvvert G - I  \vvvert > \frac12 \right) \leq \conf.
$$
\end{corollary}
When the evaluations of the function $u$ are noiseless, convergence estimates in probability with confidence level $1-\conf$ are immediate to obtain.  
If the evaluations of $u$ are noisy, then convergence estimates in probability of the form \eqref{nearoptimalprobw_2} can still be obtained by using techniques from large deviations to estimate the additional terms due to the presence of the noise, as shown in \cite{MNT2015} for standard least squares.

\subsection{Novel results}
\label{sec:opt_wei_ls_mio}
\noindent
The proof of Theorem~\ref{theo2}, and more generally the analysis in \cite{CDL,CM2016}, use a result 
from \cite{AW,Tropp} on tail bounds for sums of random matrices. We recall below this result 
from \cite[Theorem 1.1]{Tropp}, in a less general form that simplifies the presentation and  still fits our purposes. 
If $X^1,\ldots,X^\punti$ are independent $\spazio\times \spazio$ random self-adjoint and positive semidefinite matrices satisfying $\lambda_{\textrm{max}}(X^j) = \vvvert X^j \vvvert \leq R $
 almost surely and $\mathbb{E}( \sum_{j=1}^\punti X^j )=I$ then it holds 
\begin{equation}
\label{eq:tropp1}
\Pr\left(  \lambda_{\textrm{min}}\left(\sum_{j=1}^\punti X^j \right) \leq 1-\delta  \right) 
\leq \spazio \left(  \dfrac{e^{-\delta}}{(1-\delta)^{1-\delta}}  \right)^{\frac1R}, 
\quad \delta \in [0,1],
\end{equation}
\begin{equation}
\label{eq:tropp2}
\Pr\left(  \lambda_{\textrm{max}}\left(\sum_{j=1}^\punti X^j \right) \geq 1+\delta  \right) 
\leq \spazio \left(  \dfrac{e^{\delta}}{(1+\delta)^{1+\delta}}  \right)^{\frac1R}, \quad \delta \geq 0.  
\end{equation}

Since for $\delta \in (0,1)$ the upper bound in \eqref{eq:tropp2} is always greater or equal than the upper bound in \eqref{eq:tropp1}, it holds  that 
\begin{equation}
\label{eq:chernov}
\Pr\left(  \left\vvvert 
\sum_{j=1}^\punti X^j
-I \right\vvvert > \delta  \right) 
\leq 2\spazio \left(  \dfrac{e^{\delta}}{(1+\delta)^{1+\delta}}  \right)^{\frac1R}.   
\end{equation}
Finding a suitable value for $R$ and taking $\delta=\frac12$ leads to item (i) in Theorem~\ref{theo2}, see \cite{CM2016} for the proof.   

One of the features of the bounds \eqref{eq:tropp1}-\eqref{eq:tropp2} is that the matrices $X^1,\ldots,X^\punti$ need not be identically distributed. 
This property has not been exploited in the analysis in \cite{CM2016}.
The first contribution of this paper is the following Theorem~\ref{theo_mio}, which states a similar result as Theorem~\ref{theo2}, but using a different type of random samples that is very advantageous in itselft as well as for the forthcoming application to the adaptive setting. 

\begin{theorem}
\label{theo_mio}
For any  $\conf\in(0,1)$ and any integer $\spazio\geq 1$, if 
\begin{equation}
\label{eq:cond_tau_points}
\punti=\intero \spazio,
\quad 
\textrm{ with }
\intero:= \left\lceil 
\costprop^{-1} \ln\left( 
\dfrac{2\spazio}{\conf}
 \right) 
\right\rceil, 
\quad
\textrm{$\costprop$ as in \eqref{condmw}},
\end{equation} 
and 
$\sample^1, \ldots, \sample^{\spazio \intero}$ 
is a set of independent random samples such that for any $j=1,\ldots,\spazio$ the samples $\sample^{(j-1)\intero+1}, \ldots, \sample^{j \intero}$ are identically distributed according to $\newmes_j$ defined in \eqref{eq:def_mix_mes} 
then the following holds:
\begin{enumerate}
\item[{\rm (i)}] the matrix $G$ satisfies the tail bound
\begin{equation}
{\rm Pr} \, \left\{ \vvvert G- I \vvvert > \frac12 \right\} \leq
\conf; 
\label{tailhalfw}
\end{equation}
\item[{\rm (ii)}] if $u\in L^2(\setx,\errmes)$ then the estimator $u_C$ satisfies, 
\begin{equation}
\mathbb{E}(\|u-u_C\|^2)\leq \left(1+ 
\dfrac{4\costprop}{\ln( 2\spazio / \conf ) }
\right)e_{\spazio,2}(u)^2
+\conf \|u\|^2;
\label{nearoptimalcw}
\end{equation}
\item[{\rm (iii)}] with probability larger than $1-\conf$, the estimator $u_W$ satisfies
\begin{equation*}
\|u-u_W\|\leq (1+\sqrt 2) e_{\spazio,\infty,w}(u),
\end{equation*}
for all $u$ such that $\| \sqrt{w} u \|_{L^\infty}<+\infty$. 
\end{enumerate}
\end{theorem}
\begin{proof}
Proof of (i): the matrix $G$ can be decomposed as 
$G=\sum_{\indsample=1}^{ \spazio} \sum_{k=1}^{\intero} X^{\indsample k}$
where the 
$X^{\indsample k}$,  $\indsample=1,\ldots,\spazio$,  $k=1,\ldots,\intero$, are mutually independent and, given any $\indsample=1,\ldots,\spazio$,  
the $X^{\indsample 1},\ldots,X^{\indsample \intero}$ are identically distributed copies of the rank-one random matrix $X(\sample)$ defined element-wise as 
$$
X_{pq}(\sample) = \dfrac{1}{\intero \spazio} w(\sample) \basis_p(\sample) \basis_q(\sample), \quad p,q=1,\ldots,\spazio,
$$
with $\sample$ being a random variable distributed according to $\newmes_j$. 
Notice that the 
$X^{\indsample k}$,  $\indsample=1,\ldots,\spazio$,  $k=1,\ldots,\intero$,  
are not identically distributed. 
Anyway, using \eqref{eq:weight_function_def}, it holds that for any $p,q=1,\ldots,n$, 
\begin{align*}
\mathbb{E}(G_{pq}) = & 
\mathbb{E}
\left(
\sum_{\indsample=1}^{ \spazio}
\sum_{k=1}^{\intero } 
X_{pq}^{\indsample k}  
\right)  \\
 = &
\sum_{k=1}^{\intero } 
\sum_{\indsample=1}^{ \spazio}
\mathbb{E}\left(
X_{pq}^{\indsample k}  \right)  \\
= & 
\sum_{k=1}^{\intero } 
\sum_{\indsample=1}^{ \spazio}
\int_{\setx} \dfrac{1}{ \intero \spazio} w(x) \basis_p(x) \basis_q(x) \basis_\indsample(x)^2 d\errmes 
\\
= & 
\dfrac{1}{ 
\spazio}
\int_{\setx}  w(x) \basis_p(x) \basis_q(x) 
\sum_{\indsample=1}^{ \spazio}
\basis_\indsample(x)^2 d\errmes 
\\
= &
 \int_\setx \basis_p(x) \basis_q(x) d\errmes 
\\
= & \delta_{pq},
\end{align*}
and therefore $\mathbb{E}(G)=I$.  
We then use \eqref{eq:chernov} to obtain that if 
$
\vvvert X^{\indsample k}(\sample) \vvvert \leq R  
$
almost surely for any $\indsample=1,\ldots,\spazio$ and any $k=1,\ldots,\intero$
then for any $\delta\in(0,1)$ it holds 
$$
\Pr\left(  \vvvert  G-I  \vvvert > \delta  \right) \leq 2 \spazio \exp\left( - \frac{
c_\delta
}{R}   \right),  
$$
with $c_\delta:=(1+\delta) \ln(1+\delta) - \delta>0$. 
We choose $\delta=\frac12$ and obtain $c_{\frac12}=\costprop$
as in \eqref{condmw}. 
Since $X^{\indsample k}$ has rank one and
$$ \vvvert X^{\indsample k}(\sample)  \vvvert^2 =
\textrm{trace}\left( (X^{\indsample k}(\sample))^\top  X^{\indsample k}(\sample) \right)
=
\left(
 \frac{1}{\intero \spazio} w(x) \sum_{ \ell =1}^\spazio \basis_\ell(x)^2 
\right)^2
= \frac{1}{\intero^2}
$$ 
for all $\indsample=1,\ldots,\spazio$, for all $k=1,\ldots,\intero$ and uniformly for all $\sample \in \setx$, we can take $R=1/\intero$ and obtain that, if $\punti$ and $\spazio$ satisfy \eqref{eq:cond_tau_points} then 
\begin{align*}
\Pr\left(  \vvvert   G-I \vvvert > \frac12   \right)   \leq   2\spazio e^{ - \costprop \intero }       
 \leq   2\spazio e^{ - \ln( 2\spazio/\conf ) }  = \conf.     
\end{align*}

The overall structure of the proof of ii) follows \cite{CM2016}, with some differences due to the fact that here the samples $\sample^1,\ldots,\sample^\punti$ are not identically distributed.
First we identify the underlying probability measure associated to these samples.
The $\punti=\intero \spazio$ samples $\sample^1,\ldots,\sample^\punti$ are all mutually independent, 
and are subdivided into $\intero$ blocks, where each block contains $\spazio$ random samples. 
More precisely, each block contains one random sample distributed as 
$\newmes_j$, 
for $j=1,\ldots,\spazio$. 
The probability measure of each block $z=(z^1,\ldots,z^\spazio)$ is 
$
d\mymes_\spazio := \prod_{j=1}^\spazio |\basis_j(z^j)|^2 \, d\errmes,  
$
where each $z^j \in \setx$.  
The probability measure of $\intero$ blocks,  
with all the $\intero \spazio$ random samples $\sample^1,\ldots,\sample^\punti$, 
is 
$
d\mymes^\punti := \otimes^\intero d\mymes_\spazio. 
$
Let $\Omega$ be the set of all possible draws from $\mymes^\punti$, $\Omega_+$ be the set  of all draws such that 
$$
\vvvert G-I \vvvert \leq \frac12, 
$$
and  $\Omega_- := \Omega \setminus \Omega_+$ be its complement.
Under the assumptions of Theorem~\ref{theo_mio}, from  \eqref{tailhalfw} it holds that 
$$
\Pr(\Omega_-) = \int_{\Omega_-} d\mymes^\punti \leq \conf. 
$$

Denote $g:=u-\Pi_\spazio u$. We consider the event $\vvvert G-I \vvvert \leq \frac12$, where it holds 
$$
\| u-u_C \|^2 = \| u- u_W \|^2 = \| g  \|^2 + \| \Pi_\spazio^\punti g \|^2, 
$$
since $\Pi_\spazio^\punti \Pi_\spazio u = \Pi_\spazio u$ and $g$ is orthogonal to $V$. 
Denoting with $(a_1,\ldots,a_\spazio)^\top$ the solution to the linear system 
$
G a = b, 
$ 
and $b=(\langle g, \basis_k \rangle_\punti )_{ k=1,\ldots,\spazio} $ we have that  
$$
\| u-u_C \|^2 = e_{\spazio,2}(u)^2 +  \sum_{k=1}^\spazio | a_k  |^2.
$$
Since $\vvvert G-I \vvvert \leq \frac12 
\implies 
\vvvert G \vvvert 
\geq \frac12
\implies 
\vvvert G^{-1} \vvvert \leq 2, 
$
from the line above 
$$
\| u - u_C \|^2  \leq e_{\spazio,2}(u)^2 + 4 \sum_{k=1}^\spazio | \langle  g,\basis_k \rangle_\punti |^2.   
$$
In the event $\vvvert G-I \vvvert> \frac12$ by the definition of $u_C$ in \eqref{eq:def_cond_est} we have $\| u-u_C\|=\| u \|$.
Taking the expectation of $\| u - u_C \|^2$ w.r.t.~$\mymes^\punti$ we obtain that 
\begin{align*}
\mathbb{E}(\| u - u_C \|^2 )
= & 
\int_{\Omega_+} 
\| u - u_C \|^2 \, d\mymes^\punti 
+
\int_{\Omega_-} 
\| u - u_C \|^2 \, d\mymes^\punti 
\\
 \leq & \left( e_{\spazio,2}(u)^2 + 4 \sum_{k=1}^\spazio 
\mathbb{E}( | \langle  g,\basis_k \rangle_\punti |^2  ) 
\right) \Pr(\Omega_+)
+ 
\| u \|^2
\Pr(\Omega_-)
\\
 \leq &  e_{\spazio,2}(u)^2 + 4 \sum_{k=1}^\spazio 
\mathbb{E}( | \langle  g,\basis_k \rangle_\punti |^2  ) 
+ 
\conf \| u \|^2. 
\end{align*}

We now study the second term in the above expression, crucially exploiting 
the structure of the random samples and the fact that their 
expectations still pile up and simplify,  
despite the samples are not identically distributed:   
\begin{align*}
\sum_{k=1}^\spazio 
\mathbb{E}( | \langle  g,\basis_k \rangle_\punti |^2  )
= & 
\sum_{k=1}^\spazio   
\mathbb{E}
\left( 
\dfrac{1}{\punti^2}
\sum_{i=1}^\punti
\sum_{j=1}^\punti
w(x^i) w(x^j) g(x^i) g(x^j) \basis_k(x^i) \basis_k(x^j)
\right)
\\
= & 
\dfrac{1}{\punti^2}
\sum_{k=1}^\spazio   
\sum_{i=1}^\punti
\sum_{j=1}^\punti
\mathbb{E}
\left( 
w(x^i) w(x^j) g(x^i) g(x^j) \basis_k(x^i) \basis_k(x^j)
\right)
\\
= &
\dfrac{1}{\punti^2}
\left(
\sum_{k=1}^\spazio   
\underbrace{
\sum_{i=1}^\punti
\sum_{j=1\atop j\neq i}^\punti
\mathbb{E}
\left( 
w(x^i) w(x^j) g(x^i) g(x^j) \basis_k(x^i) \basis_k(x^j)
\right)
}_{I}
+
\underbrace{
\sum_{k=1}^\spazio   
\sum_{i=1}^\punti
\mathbb{E}
\left( 
\left(
w(x^i) g(x^i) \basis_k(x^i) 
\right)^2
\right)
}_{II}
\right)
. 
\end{align*}
For term I: with any $\iteraz=1,\ldots,\spazio$, using in sequence the independence of the samples, the structure of the samples and the definition of $w$ we obtain
\begin{align*}
I = &    
\sum_{i=1}^\punti
\sum_{j=1\atop j\neq i}^\punti
\mathbb{E}
\left( 
w(x^i) g(x^i) \basis_k(x^i) 
\right)
\mathbb{E}
\left( 
 w(x^j)  g(x^j)  \basis_k(x^j)
\right)
\\
= &
\sum_{i=1}^\punti
\mathbb{E}
\left( 
w(x^i) g(x^i) \basis_k(x^i) 
\right)
\sum_{j=1\atop j\neq i}^\punti
\mathbb{E}
\left( 
 w(x^j)  g(x^j)  \basis_k(x^j)
\right)
\\
= &
\sum_{i=1}^\punti
\mathbb{E}
\left( 
w(x^i) g(x^i) \basis_k(x^i) 
\right)
\left( 
\sum_{j=1}^\punti
\mathbb{E}
\left( 
 w(x^j)  g(x^j)  \basis_k(x^j)
\right)
- 
\mathbb{E}
\left( 
 w(x^i)  g(x^i)  \basis_k(x^i)
\right)
\right)
\\
= &
\sum_{i=1}^\punti
\mathbb{E}
\left( 
w(x^i) g(x^i) \basis_k(x^i) 
\right)
\left( 
\sum_{\ell=1}^\intero
\sum_{j=1}^\spazio
\int_\setx
w(x)
 g(x)  \basis_k(x)
\basis_j(x)^2
d\errmes
- 
\mathbb{E}
\left( 
 w(x^i)  g(x^i)  \basis_k(x^i)
\right)
\right)
\\
= &
\sum_{i=1}^\punti
\mathbb{E}
\left( 
w(x^i) g(x^i) \basis_k(x^i) 
\right)
\left( 
\sum_{\ell=1}^\intero
\int_\setx
w(x)
 g(x)  \basis_k(x)
\sum_{j=1}^\spazio
\basis_j(x)^2
d\errmes
- 
\mathbb{E}
\left( 
 w(x^i)  g(x^i)  \basis_k(x^i)
\right)
\right)
\\
= &
\sum_{i=1}^\punti
\mathbb{E}
\left( 
w(x^i) g(x^i) \basis_k(x^i) 
\right)
\left( 
\intero \spazio 
\underbrace{\int_\setx
 g(x)  \basis_k(x)
d\errmes
}_{=0 }
- 
\mathbb{E}
\left( 
 w(x^i)  g(x^i)  \basis_k(x^i)
\right)
\right)
\\
= &
-
\sum_{i=1}^\punti
\left(
\mathbb{E}
\left( 
w(x^i) g(x^i) \basis_k(x^i) 
\right)
\right)^2  <0, 
\end{align*}
where $\langle g, \basis_k \rangle=0$ for all $k=1,\ldots,n$ because $g$ is orthogonal to $V$.  
For term II, again exploiting the structure of the samples and the definition of $w$ it holds    
\begin{align*}
II = & 
\sum_{i=1}^\punti
\mathbb{E}
\left( 
w(x^i)^2 g(x^i)^2
\sum_{k=1}^\spazio   
 \basis_k(x^i)^2 
\right)
\\
= &
\spazio
\sum_{i=1}^\punti
\mathbb{E}
\left( 
w(x^i) g(x^i)^2
\right)
\\
= &
\spazio
\sum_{j=1}^\intero
\sum_{k=1}^\spazio
\int_\setx
w(x)
g(x)^2  \basis_k(x)^2 d\errmes
\\
= &
\spazio
\sum_{j=1}^\intero
\int_\setx
w(x)
g(x)^2  
\sum_{k=1}^\spazio
\basis_k (x)^2 d\errmes
\\
= &
\spazio^2
\intero 
\int_\setx
g(x)^2   d\errmes
\\
= &
\spazio^2 \intero
\|
g
\|^2. 
\end{align*}

Putting the pieces together, replacing $\punti=\intero \spazio$ in term II and neglecting the nonpositive contribution of term I, we obtain 
$$
\mathbb{E}(\| u - u_C \|^2 ) \leq \left( 1+ \dfrac{ 4\spazio }{ \punti }  \right) e_{\spazio,2}(u)^2 +
  \conf \| u \|^2.  
$$
Since $\spazio/\punti=\intero^{-1} \leq \costprop/\ln(2\spazio/\conf)$ we finally obtain 
\eqref{nearoptimalcw}. 

The proof of iii) uses i) and then proceeds in the same way as for the proof of iii) in Theorem~\ref{theo2} from \cite{CM2016}. 
From the definition of the spectral norm 
$$
\vvvert G-I \vvvert \leq \frac12 \iff \frac12 \| v \|^2 \leq \| v \|_\punti^2 \leq \frac32 \|v \|^2, \quad v \in V, 
$$
and this norm equivalence holds at least with probability $1-\conf$ 
from item (i) under condition \eqref{eq:cond_tau_points}. 
Using the above norm equivalence, the Pythagorean identity
$\| u - v \|_\punti^2 = \| v - u_W \|_\punti^2
+  \| u - u_W  \|_\punti^2
$,  
and 
$\max( \| u -v \|_\punti , \| u -v \| )\leq \| \sqrt w ( u-v) \|_{L^\infty}$, 
for any $v \in V$ it holds that 
\begin{align*}
\| u-u_W \| \leq &
\|  u - v\| + \| v - u_W \| 
\\
\leq &
\| u - v \| + \sqrt2 \| v - u_W \|_\punti
\\
\leq &
\| u - v \| + \sqrt2 \| u - v \|_\punti   
\\
\leq &
(1+\sqrt2) \| \sqrt{w}( u - v) \|_{L^\infty}. 
\end{align*}
Since $v$ is arbitrary we obtain the thesis.  
\end{proof}

The next Corollary~\ref{coro_theo_mio}
extends 
Theorem~\ref{theo_mio}
to any $\punti$ (not necessarily an integer multiple of $\spazio$) 
satisfying \eqref{eq:cond_alpha_k}.
For any (fixed) $\intero=0,\ldots \lfloor \punti/\spazio \rfloor$, 
the set of $\punti$ random samples in Corollary~\ref{coro_theo_mio} is obtained by merging $\punti-\intero \spazio$ 
random samples distributed as $\auxmes_\spazio$ 
and $\intero$ random samples from $\newmes_j$ for all $j=1,\ldots,\spazio$.    
When 
$\punti$ is an integer  multiple  of $\spazio$ and $\intero=\lfloor \punti/\spazio \rfloor$, 
Corollary~\ref{coro_theo_mio} 
gives 
Theorem~\ref{theo_mio} as a particular case. 
When $\intero=0$, all the random samples in  Corollary~\ref{coro_theo_mio} 
are distributed as  $\auxmes_\spazio$, like in Corollary~\ref{thm:coro_alpha}.

\begin{corollary}
\label{coro_theo_mio}
For any  $\conf\in(0,1)$, any integers $\spazio\geq 1$, $\punti \geq \spazio$ and $\intero=0,\ldots,\lfloor \punti/\spazio \rfloor$,  
if 
$\punti$ satisfies \eqref{eq:cond_alpha_k}
and 
$\sample^1, \ldots, \sample^{\punti}$ 
is a set of independent random samples such that for any $j=1,\ldots,\spazio$ 
the $\sample^{(j-1)\intero+1}, \ldots, \sample^{j \intero}$ are identically distributed according to $\newmes_j$ defined in \eqref{eq:def_mix_mes}, 
and 
the $\sample^{\spazio 
\intero
+1  },\ldots,\sample^{\punti }$ are identically distributed as 
$\auxmes_\spazio$,   
then items i), ii) and iii) of Theorem~\ref{theo_mio} hold true. 
\end{corollary}
\begin{proof}
We proceed as in the proof of Theorem~\ref{theo_mio} but using the decomposition  
$G=\sum_{\indsample=1}^{ \spazio} \sum_{k=1}^{\intero} X^{\indsample k} + \sum_{\ell= \spazio \intero +1 }^{\punti} X^{\ell} $, where all the $X^{\indsample k}$ and $X^\ell$ are mutually independent, 
and given any $\indsample=1,\ldots,\spazio$, the $X^{\indsample 1},\ldots,X^{\indsample \intero}$ are identically distributed copies of the rank-one random matrix $X(\sample)$ defined element-wise as 
$$
X_{pq}(\sample) = \dfrac{1}{
\punti
} w(\sample) \basis_p(\sample) \basis_q(\sample), \quad p,q=1,\ldots,\spazio,
$$
with $\sample$ being a random variable distributed according to $\newmes_j$, 
and the $X^\ell$ 
are identically distributed copies of 
$X(\sample)$ 
but with $\sample$ being a random variable distributed according to $\auxmes_\spazio$.  
For any $p,q=1,\ldots,n$, 
$\intero=0,\ldots,\lfloor \punti/\spazio \rfloor$,
\begin{align*}
\mathbb{E}(G_{pq}) = & 
\sum_{k=1}^{\intero } \sum_{\indsample=1}^{ \spazio} \mathbb{E}\left( X_{pq}^{\indsample k}  \right)  
+
\sum_{\ell=\spazio \intero +1}^{\punti } 
\mathbb{E}\left( X_{pq}^{\ell}  \right)  
\\
= & 
\sum_{k=1}^{\intero }  \sum_{\indsample=1}^{ \spazio} \int_{\setx} \dfrac{1}{\punti } w(\sample) \basis_p(\sample) \basis_q(\sample) \basis_\indsample(\sample)^2 d\errmes 
+
\sum_{\ell=\spazio \intero +1}^{\punti }   \int_{\setx} \dfrac{1}{\punti } w(\sample) \basis_p(\sample) \basis_q(\sample)
\dfrac{ \sum_{\indsample=1}^{ \spazio} \basis_\indsample(x)^2}{\spazio} d\errmes 
\\
= & 
\dfrac{\spazio \intero}{ 
\punti}
\int_{\setx}  
\basis_p(\sample) \basis_q(\sample) 
d\errmes 
+
\dfrac{\punti - \spazio \intero}{ 
\punti}
\int_{\setx}  
\basis_p(\sample) \basis_q(\sample) 
d\errmes 
= 
\delta_{pq}. 
\end{align*}
When $\intero=0$ ($\punti$ is an integer multiple of $\spazio$ and $\intero=\lfloor \punti/\spazio \rfloor$), 
the leftmost (rightmost) sum in the first equation above is empty.   
Since 
$$ \vvvert X^{\indsample k}(\sample)  \vvvert 
= 
\vvvert X^{\ell}(\sample)  \vvvert
= 
\frac{\spazio}{\punti}
$$ 
for all $\indsample=1,\ldots,\spazio$, for all $k=1,\ldots,\intero$, 
for all $\ell=\spazio \intero +1,\ldots, \punti$ 
 and uniformly for all $\sample \in \setx$, we can take $R=\spazio/\punti$ and obtain that if $\punti$ and $\spazio$ satisfy 
\eqref{eq:cond_alpha_k} 
then \eqref{tailhalfw} 
holds true. 

For the proof of 
\eqref{nearoptimalcw}, 
we proceed as in 
the proof of 
Theorem~\ref{theo_mio} 
with two differences. When bounding term I, we split the random samples as  
\begin{align*}
\sum_{ \indsample =1 }^{ \punti } \mathbb{E}( w(\sample^\indsample) g(\sample^\indsample) \basis_k(\sample^\indsample) ) = &
\sum_{\ell=1}^{\intero }  \sum_{\indsample=1}^{ \spazio} \int_{\setx} 
w(\sample) g(\sample) \basis_k(\sample) \basis_\indsample(\sample)^2 d\errmes 
+
\sum_{\ell=\spazio \intero +1}^{\punti }   \int_{\setx} 
w(\sample) g(\sample) \basis_k(\sample)
\dfrac{ \sum_{\indsample=1}^{ \spazio} \basis_\indsample(\sample)^2}{\spazio} d\errmes \\
= & \spazio \intero 
\int_{\setx} 
 g(\sample) \basis_k(\sample) 
d\errmes 
+
(\punti - \spazio \intero) 
\int_{\setx} 
g(\sample) \basis_k(\sample)
d\errmes = 0,  
\end{align*} 
and this term again vanishes due to the orthogonality of $g$ to $\basis_k$. 
For term II, splitting again the random samples 
we obtain 
\begin{align*}
\sum_{k=1}^{\spazio}
\sum_{\indsample=1}^{\punti}
\mathbb{E}\left( \left( w(\sample^\indsample) g(\sample^\indsample) \basis_k(\sample^\indsample) \right)^2 \right)
= &
\sum_{\indsample=1}^{\punti}
\mathbb{E}\left(  w(\sample^\indsample)^2 g(\sample^\indsample)^2 
\sum_{k=1}^{\spazio}
\basis_k(\sample^\indsample)^2 \right)
=
\spazio
\sum_{\indsample=1}^{\punti}
\mathbb{E}\left(  
w(\sample) g(\sample)^2 
\right)
\\
= &
\spazio 
\sum_{\ell=1}^{\intero}
\sum_{\indsample=1}^{\spazio}
\int_{\setx} 
w(\sample) g(\sample)^2 \basis_j(\sample)^2 
d\errmes 
+
\spazio 
\sum_{\ell=\spazio \intero + 1}^{\punti}
\int_{\setx}   
w(\sample) g(\sample)^2 
\dfrac{
\sum_{k=1}^{\spazio} \basis_k(\sample)^2 
}{\spazio}
d\errmes 
\\
= & \spazio^2 \intero \|g \|^2 + \spazio (\punti - \spazio \intero) \|g \|^2 =   \spazio \punti \| g\|^2. 
\end{align*}

The proof of the last item
is the same as the corresponding proof in Theorem~\ref{theo_mio}, 
but using 
\eqref{tailhalfw}
with the random samples of Corollary~\ref{coro_theo_mio}.  
\end{proof}

\section{Adaptive approximation with a nested sequence of spaces }
\label{sec:adaptive}
\noindent
We now apply the results for a given approximation space from the previous section to an arbitrary sequence of nested spaces $(V_{\iteraz})_{\iteraz\geq1} \subset L^2(\setx,\errmes)$, with $V_\iteraz:=\textrm{span}\{\basis_1,\ldots,\basis_{\spazio_\iteraz}\}$ and $\spazio_\iteraz:=\textrm{dim}(V_\iteraz)$. 
Theorem~\ref{theo2} and Theorem~\ref{theo_mio} provide two different approaches to build the set of random samples for a given approximation space, and each one of them can be applied to the adaptive setting. 
Since the samples are adapted to the space, the underlying challenge is how to recycle as much as possible the samples associated to spaces from the previous iterations, 
in order to keep the overall number of generated samples from iteration one to $\iteraz$ of the same order as $\textrm{dim}(V_\iteraz)$, \emph{i.e.}~the same scaling as in the results for an individual approximation space. 

First we briefly discuss the approach using Theorem~\ref{theo_mio}. 
This theorem prescribes the precise number of random samples coming from each component of the mixture \eqref{eq:changemes} associated to the space. 
When the spaces are nested, this trivially allows one to recycle all the samples from all the previous iterations, just by adding the missing samples to the previous ones, as shown in \eqref{eq:decomposition_mes_chi}.
The concrete procedure and the related Algorithm~\ref{algo:sampling_mio_simplif} are explained in Section~\ref{sec:det_seq_samp}.  

The approach using Theorem~\ref{theo2} is not as simple and effective as the previous one. 
Without recycling the samples from the previous iterations, the na{\"i}ve sequential application of Theorem~\ref{theo2} to each space $V_1,\ldots,V_\iterazmax$ would require the generation of an overall number of samples equal to $\sum_{\iteraz=1}^{\iterazmax} \punti_\iteraz$, with $\punti_{\iteraz}$ samples drawn from each $\auxmes_{\spazio_\iteraz}$. 
However, despite $\auxmes_{\spazio_\iteraz}$ changes at each iteration $\iteraz$, it is possible to recycle most, but not all, of the samples from the previous iterations by leveraging the additive structure of $\auxmes_{\spazio_\iteraz}$ as in \eqref{eq:decomposition_mixture}. 
This procedure is described in Section~\ref{sec:seq_random_sampling} together with Algorithm~\ref{algo:sampling_altro}.

The next results are obtained by applying Theorem~\ref{theo_mio} (respectively Theorem~\ref{theo2}) individually for each space $V_\iteraz$ and using a union bound, with the random samples produced by  Algorithm~\ref{algo:sampling_mio_simplif} (respectively Algorithm~\ref{algo:sampling_altro}). 
Here $I_\iteraz\in\mathbb{R}^{\spazio_\iteraz \times \spazio_\iteraz}$ denotes the identity matrix. 
For any $\paramzeta >1$,  $\zeta(\paramzeta)$ denotes the Riemann zeta function. 
The best approximation error \eqref{eq:def_best_approx_l2} of $u$ on the space $V_\iteraz$ is denoted by $e_{\spazio_\iteraz,2}(u)$, and $u_C^\iteraz$ denotes the estimator  
\eqref{eq:def_cond_est} on $V_\iteraz$. 

\begin{theorem}
\label{thm:teo_union_bounds_mio}
Let $\conf \in (0,1)$, $\paramzeta >1$ be real numbers 
and  
$\iterazmax \geq 1$ be an integer.  
Given any nested sequence of spaces $V_1 \subset \ldots \subset V_\iterazmax
\subset L^2(\setx,\errmes)$ with dimensions $\spazio_1  < \ldots < \spazio_\iterazmax$, if 
\begin{equation}
\label{eq:condition_points_zeta_tau}
\punti_\iteraz 
= \intero_\iteraz \spazio_\iteraz, \quad 
\intero_\iteraz:= 
\left\lceil
\costprop^{-1} \ln\left( 
\dfrac{ 
\zeta(\paramzeta) \, 
\spazio_\iteraz^{\paramzeta +1} 
}{ \conf  }
\right)
\right\rceil
, \qquad \iteraz=1,\ldots,\iterazmax,  
\end{equation}
then 
\begin{enumerate}
\item[{\rm (i)}] 
$$
\textrm{Pr}\left( 
\bigcap_{\iteraz=1}^{\iterazmax}
\left\{
\vvvert G_\iteraz - I_\iteraz \vvvert \leq \frac12
\right\}
\right) \geq 1-\conf,  
$$
where $G_\iteraz \in \mathbb{R}^{\spazio_\iteraz \times \spazio_\iteraz}$ is defined element wise as
$$
(G_\iteraz)_{ pq } = 
\punti_\iteraz^{-1} \sum_{\indsample=1}^{\punti_\iteraz}
\basis_p \left( \sample^\indsample \right)
\basis_q \left( \sample^\indsample
\right), 
$$
and $\sample^1,\ldots,\sample^{\punti_\iterazmax}$ is a set of $\punti_\iterazmax$ independent random samples such that for any $\iteraz=1,\ldots,\iterazmax$ and for any $j=1,\ldots,\spazio_\iteraz$  the samples $\sample^{(j-1)\intero_\iteraz +1}, \ldots, \sample^{j \intero_\iteraz}$ are distributed according to $\newmes_j$. The set $\sample^1,\ldots,\sample^{\punti_\iterazmax}$ can be generated by Algorithm~\ref{algo:sampling_mio_simplif}.  

\item[{\rm (ii)}] If $u\in L^2(\setx,\errmes)$ then for any $\iteraz=1,\ldots,\iterazmax$ the estimator $u_C^\iteraz$ satisfies, 
\begin{equation*}
\mathbb{E}(\|u-u_C^\iteraz \|^2)\leq \left(1+ \dfrac{4\costprop}{\ln( \zeta(\paramzeta) \spazio_\iteraz^{\paramzeta+1} /\conf )}  \right) e_{\spazio_\iteraz,2}(u)^2 + \conf \|u\|^2.
\end{equation*}
\end{enumerate}

\end{theorem}
\begin{proof}
Proof of (i). 
From Lemma~\ref{thm:teo_algo_red}, for any $\iteraz=1,\ldots,\iterazmax$ Algorithm~\ref{algo:sampling_mio_simplif} with $\intero_\iteraz$ as in \eqref{eq:condition_points_zeta_tau} generates a set $\sample^1,\ldots,\sample^{\punti_\iteraz}$ of $\punti_\iteraz$ random samples 
with the required properties.  
By construction, these random samples satisfy the assumptions of Theorem~\ref{theo_mio}, and are used to compute the matrix $G_\iteraz$.
For any $\iteraz=1,\ldots,\iterazmax$, using Theorem~\ref{theo_mio} individually for each $G_\iteraz$ with 
$$
\conf_\iteraz=\dfrac{   \conf  }{ \zeta(\paramzeta) \, \spazio_\iteraz^\paramzeta}, 
$$ 
gives 
$$
\sum_{\iteraz=1}^\iterazmax
\textrm{Pr}\left( 
\left\{
\vvvert G_\iteraz - I_\iteraz \vvvert > \frac12
\right\}
\right) \leq 
 \sum_{\iteraz=1}^\iterazmax \conf_\iteraz
\leq 
\dfrac{\conf}{\zeta(\paramzeta)}
\sum_{\iteraz= 1}^{\iterazmax} \dfrac{1}{\iteraz^\paramzeta}
\leq
\dfrac{\conf}{\zeta(\paramzeta)}
\sum_{\iteraz\geq 1}
\dfrac{1}{\iteraz^\paramzeta}
= \conf, 
$$
where the second inequality follows from strict monotonicity of $(\spazio_\iteraz)_{\iteraz \geq 1}$ and $\spazio_1\geq 1$, that implies  $\spazio_\iteraz \geq \iteraz$. 

Using De Morgan's law and a probability union bound for the matrices $G_1,\ldots,G_\iterazmax$ it holds that 
\begin{align*}
\textrm{Pr}\left( 
\bigcap_{\iteraz=1}^{\iterazmax}
\left\{
\vvvert G_\iteraz - I_\iteraz \vvvert \leq \frac12
\right\}
\right) 
\geq
1- 
\sum_{\iteraz=1}^\iterazmax
\textrm{Pr}\left( 
\left\{
\vvvert G_\iteraz - I_\iteraz \vvvert > \frac12
\right\}
\right)
\geq
 1-\conf. 
\end{align*}

The proof of (ii) 
trivially follows from 
Theorem~\ref{theo_mio}, 
since $\conf_\iteraz \leq \conf$ for any $\iteraz=1,\ldots,\iterazmax$. 
\end{proof}

\begin{theorem}
\label{thm:teo_union_bounds}
Let $\conf \in (0,1)$, $\paramzeta >1$ be real numbers and $\iterazmax \geq 1$ be an integer.  
Given any nested sequence of spaces $V_1\subset  \ldots \subset V_\iterazmax \subset L^2(\setx,\errmes)$ with dimensions $\spazio_1  < \ldots < \spazio_\iterazmax$, if 
\begin{equation}
\label{eq:condition_points_zeta}
\punti_\iteraz 
=
\left\lceil
\dfrac{ \spazio_\iteraz }{\costprop}  
\ln\left( 
\dfrac{ 
\zeta(\paramzeta) \, 
\spazio_\iteraz^{\paramzeta +1} 
}{ \conf  }
\right)
\right\rceil
, \qquad \iteraz=1,\ldots,\iterazmax,  
\end{equation}
then 
\begin{enumerate}
\item[{\rm (i)}] 
$$
\textrm{Pr}\left( 
\bigcap_{\iteraz=1}^{\iterazmax}
\left\{
\vvvert G_\iteraz - I_\iteraz \vvvert \leq \frac12
\right\}
\right) \geq 1-\conf,  
$$
where $G_\iteraz \in \mathbb{R}^{\spazio_\iteraz \times \spazio_\iteraz}$ is defined element wise as
$$
(G_\iteraz)_{ pq } = 
\punti_\iteraz^{-1} \sum_{\indsample=1}^{\punti_\iteraz}
\basis_p \left( \sample^\indsample \right)
\basis_q \left( \sample^\indsample
\right), 
$$
and 
$\sample^1,\ldots,\sample^{\punti_\iterazmax}$ 
is a set 
of 
$\punti_\iterazmax$ independent random samples 
such that 
for any $\iteraz=1,\ldots,\iterazmax$   
the $\sample^1,\ldots,\sample^{\punti_\iteraz}$  
are distributed according to $\auxmes_{\spazio_\iteraz}$.  
The set $\sample^1,\ldots,\sample^{\punti_\iterazmax}$ 
 can be generated by Algorithm~\ref{algo:sampling_altro} using 
an overall number of random samples given by  
the random variable $\tilde\punti_\iterazmax$ in \eqref{eq:def_tot_numb_samples}. 

\item[{\rm (ii)}] If $u\in L^2(\setx,\errmes)$ then
for any $\iteraz=1,\ldots,\iterazmax$ 
the estimator $u_C^\iteraz$ satisfies, 
\begin{equation}
\mathbb{E}(\|u-u_C^\iteraz \|^2)\leq 
\left(
1+
\dfrac{4\costprop}{ \ln( \zeta(\paramzeta) \spazio_\iteraz^{\paramzeta+1} / \conf ) }
\right)e_{\spazio_\iteraz,2}(u)^2+
\conf
\|u\|^2.  
\label{nearoptimalcw_k_proba}
\end{equation}
\end{enumerate}
\end{theorem}
\begin{proof}
From Lemma~\ref{thm:teo_algo_pure_random}, Algorithm~\ref{algo:sampling_altro} 
generates a set $\sample^1,\ldots,\sample^{\punti_\iterazmax}$
of $\punti_\iterazmax$ random samples 
with the required properties. 
The rest of the proof 
of item (i) and (ii) 
follows the proof of Theorem~\ref{thm:teo_union_bounds_mio}, 
but applying Corollary~\ref{thm:coro_alpha} individually to each $G_\iteraz$, rather than Theorem~\ref{theo_mio}.  
\end{proof}

Condition \eqref{eq:condition_points_zeta_tau} ensures that $\punti_\iteraz$ is an integer multiple of $\spazio_\iteraz$ for any $\iteraz \geq 1$.
This condition requires a number of points $\punti_\iteraz$ only slightly larger than condition  \eqref{eq:condition_points_zeta} for the same values of $\spazio_\iteraz$, $\paramzeta$ and $\conf$ 
(the proof is postponed to the forthcoming       
Lemma~\ref{thm:lemma_equi_condi}).
However, to compare the effective number of samples used
in Theorem~\ref{thm:teo_union_bounds_mio} and Theorem~\ref{thm:teo_union_bounds} 
one cannot just compare \eqref{eq:condition_points_zeta_tau}
and \eqref{eq:condition_points_zeta}, 
because \eqref{eq:condition_points_zeta} neglects the random samples that 
have not been recycled from all the previous iterations. 
This issue is discussed in Remark~\ref{thm:remark_alghi}. 

For convenience in Remark~\ref{thm:remark_alghi}, Lemma~\ref{thm:lemma_equi_condi} and Remark~\ref{thm:surplus} we denote with $\hat \punti_\iteraz$ the number of points required by condition  
\eqref{eq:condition_points_zeta_tau},
and with $\punti_\iteraz$ the number of points required by \eqref{eq:condition_points_zeta}, for the same values of $\spazio_\iteraz$, $\paramzeta$ and $\conf$.

\begin{remark}
\label{thm:remark_alghi}
For any $\iterazmax \geq 1$, 
in Theorem~\ref{thm:teo_union_bounds_mio} 
(Theorem~\ref{thm:teo_union_bounds}) 
the set $\sample^1,\ldots,\sample^{\hat \punti_\iterazmax}$ 
($\sample^1,\ldots,\sample^{\punti_\iterazmax}$)
of random samples can be generated by Algorithm~\ref{algo:sampling_mio_simplif} (Algorithm~\ref{algo:sampling_altro}). 
In Theorem~\ref{thm:teo_union_bounds}, the generation of 
the $\punti_\iterazmax$ random samples 
requires Algorithm~\ref{algo:sampling_altro} 
to produce an overall number $\tilde\punti_\iterazmax$ 
of random samples, 
with 
$\tilde\punti_\iterazmax$ being 
the random variable 
defined in \eqref{eq:def_tot_numb_samples}. 
Among the $\tilde\punti_\iterazmax$ random samples (not necessarily distributed as $\auxmes_{\spazio_\iterazmax}$) only $\punti_\iterazmax$ 
are retained,  and the remaining $\tilde\punti_\iterazmax-\punti_\iterazmax$ are discarded.
Condition \eqref{eq:condition_points_zeta} does not take into account the $\tilde\punti_\iterazmax-\punti_\iterazmax$ discarded samples.
As a consequence, when comparing the effective number of samples required 
in Theorem~\ref{thm:teo_union_bounds_mio} and Theorem~\ref{thm:teo_union_bounds}, 
one should compare $\hat \punti_\iterazmax$  
with $\tilde\punti_\iterazmax$, and not $\hat \punti_\iterazmax$ with $\punti_\iterazmax$.

It can be shown that $\tilde\punti_\iterazmax-\punti_\iterazmax $ remains small with large probability,   
if $\punti_\iteraz$ 
satisfies \eqref{eq:condition_points_zeta} for all $\iteraz=1,\ldots,\iterazmax$.  
More precisely, 
using the upper bound $\tilde\punti_\iterazmax \leq \punti_\iterazmax +  \Gaus_\iterazmax$ with $\Gaus_\iterazmax$ defined in \eqref{eq:def_rv_gauss}, Lemma~\ref{thm:lemma_expectation_gauss_t}, the last inequality in \eqref{eq:thesis_lemma_bounds} 
and Remark~\ref{thm:gaussian_binomial_approx}, 
it can be shown that $\tilde\punti_\iterazmax$ is upper bounded by a random variable with mean $(2+\costprop) \punti_\iterazmax$ and variance $(1+\costprop) \punti_\iterazmax$ that exhibits Gaussian concentration.  
\end{remark}

\begin{lemma}
\label{thm:lemma_equi_condi}
For any $\iteraz \geq 1$ it holds that 
\begin{equation}
\label{eq:thesis_lemma_bounds}
\punti_\iteraz 
\leq 
\hat \punti_\iteraz 
\leq 
\punti_\iteraz + \spazio_\iteraz-1
\leq 
\punti_\iteraz \left( 1 + \epsilon_\iteraz  \right) -1,
\end{equation}
where  
$$
\epsilon_\iteraz := 
\costprop
\left(  \log \dfrac{ 2\spazio_\iteraz^{\paramzeta +1} \zeta(\paramzeta) }{ \conf  } \right)^{-1}
 \ll 1. 
$$
\end{lemma}
\begin{proof}
For any 
$\spazio_\iteraz \geq 1$ it holds  
\begin{align*}
\punti_\iteraz 
= 
\left\lceil  
\spazio_\iteraz
\costprop^{-1} \log \dfrac{ 2\spazio_\iteraz^{\paramzeta +1} \zeta(\paramzeta) }{ \conf  } \right\rceil
\leq 
\spazio_\iteraz
\left\lceil \costprop^{-1} \log \dfrac{ 2\spazio_\iteraz^{\paramzeta +1} \zeta(\paramzeta) }{ \conf  } \right\rceil
=
\hat \punti_\iteraz
<
\spazio_\iteraz
\left(  \costprop^{-1} \log \dfrac{ 2\spazio_\iteraz^{\paramzeta +1} \zeta(\paramzeta) }{ \conf  } +1 \right).
\end{align*}
The first inequality above proves the first inequality in \eqref{eq:thesis_lemma_bounds}. 
The rightmost strict inequality above and \eqref{eq:condition_points_zeta} prove 
the second (large) inequality in \eqref{eq:thesis_lemma_bounds}. 
The last inequality in \eqref{eq:thesis_lemma_bounds} is obtained by using once again \eqref{eq:condition_points_zeta} together with properties of the ceiling operator. 
Notice that $\costprop \approx 0.108$.
\end{proof}

\begin{remark}
\label{thm:surplus}
The small number $\hat\punti_\iteraz-\punti_\iteraz<\spazio_\iteraz$ of additional samples required by \eqref{eq:condition_points_zeta_tau} contribute to further improve the stability of $u_W$. 
This slight surplus of samples is completely negligible from a fully adaptive point of view, where at each iteration $\iteraz$ conditions \eqref{eq:condition_points_zeta_tau} or \eqref{eq:condition_points_zeta} are not necessarily fulfilled, 
but, more simply, new random samples are just added to the previous ones until a certain stability criterion is met, for example until $\vvvert G_\iteraz - I_\iteraz \vvvert \leq \thresholdcond_\iteraz$ or $\textrm{cond}(G_\iteraz)\leq \thresholdcond_\iteraz$ for some threshold $\thresholdcond_\iteraz$ eventually depending on $\iteraz$.   
\end{remark}

\begin{remark}
\label{thm:remark_improv_mix}
For any integer $\iterazmax\geq 1$ and reals $\conf \in (0,1)$, $\paramzeta >1$, 
the result in 
Theorem~\ref{thm:teo_union_bounds_mio} can be sharpened by using $\punti_\iterazmax$ random samples $\sample^1,\ldots,\sample^{\punti_\iterazmax}$ 
such that, with the same $\intero_\iteraz$ as in \eqref{eq:condition_points_zeta_tau}, 
\begin{itemize}
\item 
for any $\iteraz=1,\ldots,\iterazmax-1$ and for any $j=1,\ldots,\spazio_\iteraz$ the  $\sample^{(j-1)\intero_\iteraz +1}, \ldots, \sample^{j \intero_\iteraz}$ are distributed according to $\newmes_j$, 
\item 
at iteration $\iterazmax$, 
for any $j=1,\ldots,\spazio_\iterazmax$ the samples $\sample^{(j-1)\intero_{\iterazmax-1} +1}, \ldots, \sample^{j \intero_{\iterazmax-1}}$ are distributed according to $\newmes_j$, and 
 the samples $\sample^{\spazio_{\iterazmax}\intero_{\iterazmax-1} +1}, \ldots, \sample^{\punti_\iterazmax }$ are distributed according to $\auxmes_{\spazio_\iterazmax}$. 
\end{itemize}
For any $\iteraz=1,\ldots,\iterazmax-1$ 
the set $\sample^1,\ldots,\sample^{\punti_\iteraz}$ 
can be incrementally constructed using \eqref{eq:decomposition_mes_chi}. 
For the construction of $\sample^1,\ldots,\sample^{\punti_\iterazmax}$ 
at the last iteration $\iterazmax$:  
we recycle all the $\punti_{\iterazmax-1} = \spazio_{\iterazmax-1} \intero_{\iterazmax-1}$ samples from 
iteration $\iterazmax-1$, then we add $(\spazio_\iterazmax - \spazio_{\iterazmax -1} ) \intero_{\iterazmax-1}$ new samples, \emph{i.e.} $\intero_{\iterazmax-1}$ samples from $\newmes_{j}$ for all $j=\spazio_{\iterazmax-1}+1,\ldots,\spazio_\iterazmax$, and then add the remaining $\punti_\iterazmax - \spazio_\iterazmax \intero_{\iterazmax-1} $ samples from $\auxmes_{\spazio_\iterazmax}$.

Using the random samples 
$\sample^1,\ldots,\sample^{\punti_\iterazmax}$, 
in the proof of Theorem~\ref{thm:teo_union_bounds_mio}
we can apply Theorem~\ref{theo_mio} at the first $\iterazmax-1$ iterations, and then at iteration $\iterazmax$
apply Corollary~\ref{coro_theo_mio} with 
$\spazio=\spazio_{\iterazmax}$, 
$\punti=\punti_\iterazmax$, 
$\intero = \intero_{\iterazmax-1}$.   
This proves the same conclusions of Theorem~\ref{thm:teo_union_bounds_mio} under 
the same condition \eqref{eq:condition_points_zeta_tau} on $\punti_\iteraz$ for $\iteraz=1,\ldots,\iterazmax-1$, 
but with the slightly better condition  
\eqref{eq:condition_points_zeta}
 on $\punti_\iterazmax$, because $\punti_\iterazmax$ need not be an integer multiple of $\spazio_\iterazmax$. 
\end{remark}

\begin{remark}
Remark~\ref{thm:remark_improv_mix} can be used to develop adaptive algorithms 
that recycle all the samples from all the previous iterations, and that use a number of samples  
$\punti_\iterazmax$ given by \eqref{eq:condition_points_zeta} at the last iteration $\iterazmax$.  
Such adaptive algorithms need to detect in advance which is last iteration, in contrast to algorithms developed using Theorem~\ref{thm:teo_union_bounds_mio} where this information is not needed.  
\end{remark}

\section{Sampling algorithms}
\noindent
In the following we present two sequential algorithms that generate the random samples required by Theorem~\ref{thm:teo_union_bounds_mio} or Theorem~\ref{thm:teo_union_bounds} at any iteration say $\iterazmax$, while recyclying the samples from the previous iterations $\iteraz=1,\ldots, \iterazmax-1$. 
The algorithms are described in appendix, using the convention that a loop 
\textbf{for} $i=start$ \textbf{to} $end$ 
on the variable say $i$
is not executed  if $end<start$. 

\subsection{Deterministic sequential sampling}
\label{sec:det_seq_samp}
This section presents Algorithm~\ref{algo:sampling_mio_simplif}, that can be used to produce the random samples required by Theorem~\ref{thm:teo_union_bounds_mio} using the decomposition \eqref{eq:decomposition_mes_chi}. 
By construction, Algorithm~\ref{algo:sampling_mio_simplif} recycles all the samples from all the previous iterations. 
At any iteration $\iterazmax \geq 1$ the algorithm stores the $\punti_\iterazmax=\intero_\iterazmax \spazio_\iterazmax$ random samples in a $\spazio_\iterazmax\times \intero_\iterazmax$ matrix. 
All the elements of this matrix are modified only once as the algorithm runs from iteration one to $\iterazmax$. 
The algorithm works with any nondecreasing positive integer sequence $(\intero_\iteraz)_{\iteraz\geq 1}$. 

\begin{lemma}
\label{thm:teo_algo_red}
Let $(V_\iteraz)_{\iteraz \geq 1}$ be any sequence of nested  spaces with dimension $\spazio_\iteraz=\textrm{dim}(V_\iteraz)$, and $(\intero_\iteraz)_{\iteraz\geq 1}$ be a positive nondecreasing integer sequence. For any $\iterazmax\geq 2$, Algorithm~\ref{algo:sampling_mio_simplif} generates a set of $\punti_\iterazmax=\intero_\iterazmax \spazio_\iterazmax$ random samples $\sample^1,\ldots,\sample^{\punti_\iterazmax}$ with the property that for any $\iteraz=1,\ldots,\iterazmax$ and for any $j=1,\ldots,\spazio_\iteraz$ the samples $\sample^{(j-1)\intero_\iteraz +1}, \ldots, \sample^{j \intero_\iteraz}$ are distributed according to $\newmes_j$. 
\end{lemma}

\begin{proof}
At any iteration $\iteraz=1,\ldots,\iterazmax$ 
Algorithm~\ref{algo:sampling_mio_simplif} produces the matrix $\{ \sample^{\indsample \ell}, \, \indsample=1,\ldots,\spazio_\iteraz, \, \ell=1,\ldots, \intero_\iteraz\}$ that contains the $\intero_\iteraz \spazio_\iteraz = \punti_\iteraz$ random samples, by modifying 
only the elements 
$\{ \sample^{\indsample \ell}, \, \indsample=1,\ldots,\spazio_{\iteraz-1}, \, \ell=1+\intero_{\iteraz-1},\ldots, \intero_\iteraz\}$ and 
$\{ \sample^{\indsample \ell}, \, \indsample=1+\spazio_{\iteraz-1},\ldots,\spazio_\iteraz, \, \ell=1,\ldots, \intero_\iteraz\}$.  
By construction, for any $j=1,\ldots,\spazio_\iteraz$, the $j$th row of this matrix contains $\intero_\iteraz$ samples distributed as 
$\newmes_j$. This matrix is recasted into a column vector by means of a transposition composed with a vectorization, that piles up its rows into the vector $(\sample^1,\ldots,\sample^{\punti_\iteraz})^\top$. 
\end{proof}

When $\errmes$ is a product measure on $\setx$, random samples from all the 
$\newmes_j$ appearing in Algorithm~\ref{algo:sampling_mio_simplif} can be efficiently drawn by using the algorithms proposed in \cite{CM2016}, \emph{i.e.}~\emph{inverse transform sampling} or \emph{rejection sampling}. The computational cost required by these algorithms scales linearly with respect to $d$ and to the desired number of samples.

\subsection{Random sequential sampling}
\label{sec:seq_random_sampling}
This section presents Algorithm~\ref{algo:sampling_altro}, 
that can be used to produce the random samples 
required by Theorem~\ref{thm:teo_union_bounds}, and uses the decomposition \eqref{eq:decomposition_mixture}.
For any $\iteraz \geq 2$, a standard algorithm for generating $\punti_\iteraz$ random samples from 
$\auxmes_{\spazio_{\iteraz}}$
uses a binomial random variable 
$\binrv_\iteraz \sim \Bin\left(\punti_{\iteraz}, (\spazio_{\iteraz}-\spazio_{\iteraz-1})/\spazio_\iteraz \right)$ 
to determine the proportion of samples coming from $\sigma_{\spazio_\iteraz}$.
The first parameter of $\binrv_\iteraz$ is the number of trials, and the second parameter is the probability of success for each trial, that is given by the coefficient multiplying $d\sigma_{\spazio_\iteraz}$ in \eqref{eq:decomposition_mixture}. 
For any $\iteraz\neq\iteraz^\prime$, $\binrv_{\iteraz}$ and $\binrv_{\iteraz^\prime}$ are mutually independent.   
The amount of samples associated to $\auxmes_{\spazio_{\iteraz-1}}$ is equal to $\punti_\iteraz - \binrv_\iteraz$. 
These are the samples that the algorithm can recycle from the previous iterations, whenever necessary. 
For any $\iterazmax\geq 1$, the algorithm that generates random samples from $\auxmes_{\spazio_1},\ldots,\auxmes_{\spazio_\iterazmax}$ in a sequential manner is described in Algorithm~\ref{algo:sampling_altro}. 
Efficient algorithms have been proposed in \cite{CM2016} for drawing samples from all the probability measures $\auxmes_{\spazio_\iteraz}$ and $\sigma_{\spazio_\iteraz}$ appearing in Algorithm~\ref{algo:sampling_altro}. 
The next lemma quantifies more precisely how many unrecycled samples cumulate after say $\iterazmax$ iterations. 

\begin{lemma}
\label{thm:teo_algo_pure_random}
For any $\iterazmax\geq 1$, Algorithm~\ref{algo:sampling_altro} generates a set of $\punti_\iterazmax$ random samples $\sample^1,\ldots,\sample^{\punti_\iterazmax}$ with the property that $\sample^1,\ldots,\sample^{\punti_\iteraz}$ are distributed according to $\auxmes_{\spazio_\iteraz}$, for any $\iteraz=1,\ldots,\iterazmax$. 
The overall number of samples generated by Algorithm~\ref{algo:sampling_altro} at iteration $\iterazmax$ is  
\begin{equation}
\tilde\punti_\iterazmax := \punti_1 + \sum_{\iteraz=2}^\iterazmax 
\left( 
\binrv_\iteraz + \max\{ \punti_\iteraz - \binrv_\iteraz - \punti_{\iteraz-1}, 0 \} 
\right).  
\label{eq:def_tot_numb_samples}
\end{equation} 
\end{lemma}
\begin{proof}
The properties of the random samples are ensured by construction. 
We now proove \eqref{eq:def_tot_numb_samples}.
When $\iterazmax=1$ the sum is empty and the formula holds true. Suppose then $\iterazmax \geq 2$. 
The proof uses induction on $\iteraz$. 
At iteration $\iteraz=2$, $\punti_{\iteraz-1}$ samples from $\auxmes_{\spazio_{\iteraz-1}}$ are available, which verifies the induction hypothesis. 
Proof of the induction step: for any $\iteraz\geq 2$, supposing that  $\punti_{\iteraz-1}$ samples from $\auxmes_{\spazio_{\iteraz-1}}$ are available at iteration $\iteraz-1$, the number of recycled samples from iteration $\iteraz-1$ is $\min( \punti_\iteraz - \binrv_\iteraz,\punti_{\iteraz-1})$.
Then the algorithm adds $\max( \punti_\iteraz - \binrv_\iteraz - \punti_{\iteraz-1},0)$ new samples from $\auxmes_{\spazio_{\iteraz-1}}$. 
Afterwards $\punti_\iteraz - \max( \punti_\iteraz - \binrv_\iteraz - \punti_{\iteraz-1},0) - \min( \punti_\iteraz - \binrv_\iteraz,\punti_{\iteraz-1})$
new samples are added from $\sigma_{\spazio_\iteraz}$. At the end of iteration $\iteraz$, the algorithm produces a set containing $\punti_\iteraz$ random samples from $\auxmes_{\spazio_\iteraz}$, 
and throws away $\punti_{\iteraz-1} - \min( \punti_\iteraz - \binrv_\iteraz,\punti_{\iteraz-1})$ samples that were drawn at iteration $\iteraz-1$ from $\auxmes_{\spazio_{\iteraz-1}}$. 
Summation of each contribution of new samples at any iteration $\iteraz$ from $2$ to $\iterazmax$ gives \eqref{eq:def_tot_numb_samples}.  
\end{proof}

The number of unrecycled samples after $\iterazmax$ iterations is $\tilde\punti_\iterazmax - \punti_\iterazmax$. 
As a sum of nonnegative random variables, this number can only increase as the algorithm runs, which represents the major disadvantage of Algorithm~\ref{algo:sampling_altro}, and of any other purely random sequential algorithm.
Since $\punti_\iteraz \geq \punti_{\iteraz-1}$ and $\binrv_\iteraz\geq 0$ for all $\iteraz\geq 2$,
from \eqref{eq:def_tot_numb_samples}
we have the upper bound $\tilde\punti_\iterazmax \leq \punti_\iterazmax +  \Gaus_\iterazmax$, where $\Gaus_\iterazmax$ is the random variable defined as 
\begin{equation}
\Gaus_\iterazmax:=\sum_{\iteraz=2}^\iterazmax \binrv_\iteraz,  
\label{eq:def_rv_gauss}
\end{equation}
that gives an upper bound for the number of unrecycled samples. 
Its mean and variance are given by 
\begin{equation}
\mathbb{E}( \Gaus_\iterazmax ) 
= \sum_{\iteraz=2}^{\iterazmax} \mathbb{E}( \binrv_\iteraz ) 
= \sum_{\iteraz=2}^{\iterazmax} \punti_{\iteraz} \dfrac{\spazio_{\iteraz}-\spazio_{\iteraz-1}}{\spazio_\iteraz}, 
\label{eq:mean_binom_rv}
\quad
\textrm{Var}( \Gaus_\iterazmax ) 
= \sum_{\iteraz=2}^{\iterazmax} \textrm{Var}( \binrv_\iteraz ) 
= \sum_{\iteraz=2}^{\iterazmax} \punti_{\iteraz} \dfrac{\spazio_{\iteraz}-\spazio_{\iteraz-1}}{\spazio_\iteraz} \dfrac{\spazio_{\iteraz-1}}{\spazio_\iteraz}. 
\end{equation}
The above expressions for the mean and variance of $\Gaus_\iterazmax$ hold true for any condition between $\punti_\iteraz$ and $\spazio_\iteraz$, not necessarily of the form \eqref{eq:condition_points_zeta}. 
When \eqref{eq:condition_points_zeta} is fulfilled we have the following upper bounds. 

\begin{lemma}
\label{thm:lemma_expectation_gauss_t}
For any strictly increasing sequence $(\spazio_\iteraz)_{\iteraz\geq1}$, for any $\iterazmax\geq 2$, $\paramzeta >1$ and $\conf \in (0,1)$, if $\punti_\iteraz$ and $\spazio_\iteraz$ satisfy condition \eqref{eq:condition_points_zeta} for all $\iteraz=1,\ldots,\iterazmax$ then 
$$
\textrm{Var}(\Gaus_\iterazmax)
< 
\mathbb{E}(\Gaus_\iterazmax) 
\leq 
\punti_\iterazmax +\spazio_\iterazmax - \punti_1.
$$
\end{lemma}
\begin{proof}
\begin{align*}
\mathbb{E}( \Gaus_\iterazmax ) 
= &
\sum_{\iteraz=2}^{\iterazmax} \punti_{\iteraz}
\dfrac{\spazio_{\iteraz}-\spazio_{\iteraz-1}}{\spazio_\iteraz} 
\\
\leq &
\sum_{\iteraz =2}^\iterazmax
(\spazio_\iteraz - \spazio_{\iteraz-1})
\left\lceil
\costprop^{-1} 
\ln\left( 
\dfrac{ 
\zeta(\paramzeta) \, 
\spazio_\iteraz^{\paramzeta +1} 
}{ \conf  }
\right)
\right\rceil 
\\
\leq &
\sum_{\iteraz =2}^\iterazmax
\spazio_\iteraz 
\left\lceil
\costprop^{-1} 
\ln\left( 
\dfrac{ 
\zeta(\paramzeta) \, 
\spazio_\iteraz^{\paramzeta +1} 
}{ \conf  }
\right)
\right\rceil
- \spazio_{\iteraz-1}
\left\lceil
\costprop^{-1} 
\ln\left( 
\dfrac{ 
\zeta(\paramzeta) \, 
\spazio_{\iteraz-1}^{\paramzeta +1} 
}{ \conf  }
\right)
\right\rceil
\\
= &
\spazio_\iterazmax 
\left\lceil
\costprop^{-1} 
\ln\left( 
\dfrac{ 
\zeta(\paramzeta) \, 
\spazio_\iterazmax^{\paramzeta +1} 
}{ \conf  }
\right)
\right\rceil
- \spazio_{1}
\left\lceil
\costprop^{-1} 
\ln\left( 
\dfrac{ 
\zeta(\paramzeta) \, 
\spazio_{1}^{\paramzeta +1} 
}{ \conf  }
\right)
\right\rceil
\\ 
\leq & \punti_{\iterazmax} +\spazio_\iterazmax - \punti_{1}. 
\end{align*}
The inequality for $\textrm{Var}(\Gaus_\iterazmax)$ follows from \eqref{eq:mean_binom_rv} and strict monotonicity of the sequence $(\spazio_\iteraz)_\iteraz$.
\end{proof}

\begin{remark}
\label{thm:gaussian_binomial_approx}
Here we show  that the random variable  $\Gaus_\iterazmax$ concentrates like a Gaussian random variable with mean and variance given by \eqref{eq:mean_binom_rv}.
The central limit theorem for a binomial random variable $\binrv\sim\Bin(m,p)$ with number of trials $m$ and success probability $p$ states that 
\begin{align*}
\lim_{m\to \infty} \textrm{Pr}\left(  \dfrac{ \binrv - mp }{\sqrt{ m p (1-p)}} \leq b \right) 
= & 
\Phi\left(b \right), \quad b \in \mathbb{R},
\end{align*}
where $\Phi$ is the cumulative distribution function of the standard Gaussian distribution. 
This justifies the well-known Gaussian approximation of $\binrv$ 
when $m$ is sufficiently large. 
This 
approximation is 
very accurate already when $mp\geq 5$ and $m(1-p)\geq 5$. 
In our settings, when $\punti_\iteraz$ and $\spazio_\iteraz$ satisfy \eqref{eq:condition_points_zeta} for some $\conf\in(0,1)$ and $\paramzeta>1$, the parameters of the binomial random variables $\binrv_\iteraz$ overwhelmingly verify the above conditions for any $\iteraz=2,\ldots,\iterazmax$, since 
\begin{equation}
\label{eq:first_hypo}
\dfrac{\punti_\iteraz (\spazio_\iteraz - \spazio_{\iteraz -1}) }{\spazio_\iteraz} 
\geq 
\costprop^{-1}
\ln\left( 
\dfrac{ 
\zeta(\paramzeta) \, 
\spazio_\iteraz^{\paramzeta +1} 
}{ \conf  }
\right)
(\spazio_\iteraz - \spazio_{\iteraz-1})
\geq
\costprop^{-1}
\ln\left( 
\dfrac{ 
\zeta(\paramzeta) \, 
\spazio_\iteraz^{\paramzeta +1} 
}{ \conf  }
\right)
 \gg  5, 
\end{equation}
\begin{equation}
\dfrac{\punti_\iteraz \spazio_{\iteraz -1} }{\spazio_\iteraz} =  
\costprop^{-1}
\ln\left( 
\dfrac{ 
\zeta(\paramzeta) \, 
\spazio_\iteraz^{\paramzeta +1} 
}{ \conf  }
\right)
\spazio_{\iteraz-1}
 \gg
5, 
\label{eq:assumpt_approx_bin_gaus}
\end{equation}
and $\costprop^{-1} \approx 9.242$. 
Using the Gaussian approximation of the binomial distribution, each $\binrv_\iteraz$ behaves like a Gaussian r.v.~with the same mean and variance. 
A finite linear combination of independent Gaussian random variables is a Gaussian random variable. 
Hence the r.v.~$\Gaus_\iterazmax$ behaves like a Gaussian r.v.~with mean and variance as in \eqref{eq:mean_binom_rv}. 
\end{remark}

\subsection{Comparison of the sampling algorithms}
\label{sec:comparison}
The main properties of Algorithm~\ref{algo:sampling_mio_simplif} and Algorithm~\ref{algo:sampling_altro} are resumed below. 
At any iteration say $\iterazmax$: 
\begin{itemize}
\item 
Algorithm~\ref{algo:sampling_mio_simplif} generates $\punti_\iterazmax = \intero_\iterazmax \spazio_\iterazmax$ independent random samples with $\intero_\iterazmax$ being any positive integer, and such that $\intero_\iterazmax$ of these random samples are drawn from $\newmes_j$, for any $j=1,\ldots,\spazio_{\iterazmax}$.    
This algorithm recycles all the samples generated at all the previous iterations $\iteraz=1,\ldots,\iterazmax-1$. 

\item 
Algorithm~\ref{algo:sampling_altro} generates $\punti_\iterazmax$ independent random samples 
from $\auxmes_{\spazio_\iterazmax}$. 
This algorithm recycles most of the samples generated at all the previous iterations $\iteraz=1,\ldots,\iterazmax-1$.  
If \eqref{eq:condition_points_zeta} holds true at any iteration, then the number of unrecycled samples at iteration $\iterazmax$ is upper bounded by the random variable \eqref{eq:def_rv_gauss} with mean $(1+\costprop)\punti_\iterazmax$ and variance $(1+\costprop)\punti_\iterazmax$, that exhibits Gaussian concentration. 
    
\item 
During the execution, 
Algorithm~\ref{algo:sampling_mio_simplif} modifies each element of the output set 
$\sample^1,\ldots,\sample^{\punti_\iterazmax}$ only once, in contrast to 
Algorithm~\ref{algo:sampling_altro} that can modify the same element several times, 
when discarding previously generated random samples. 

\item Algorithm~\ref{algo:sampling_mio_simplif} and Algorithm~\ref{algo:sampling_altro} use any  sequence of nested spaces $(V_\iteraz)_{\iteraz}$.

\item 
The weighted least-squares estimators constructed with the random samples generated by both Algorithm~\ref{algo:sampling_mio_simplif} and Algorithm~\ref{algo:sampling_altro} share the same theoretical guarantees, 
see Theorem~\ref{thm:teo_union_bounds_mio} and Theorem~\ref{thm:teo_union_bounds}. 

\item 
In practice Algorithm~\ref{algo:sampling_mio_simplif} outperforms Algorithm~\ref{algo:sampling_altro} in all our numerical tests, recycling all the samples from all the previous iterations, and producing on average more stable Gramian matrices. 
\end{itemize}

\begin{remark}
\label{thm:stabil_localiz}
When using random samples from $\mymes^\punti$ rather than from $\auxmes^\punti$, 
the benefits of variance reduction increase with more localized basis than orthogonal polynomials, like wavelets. 
The structure of the random samples from $\mymes^\punti$ ensures that for any element of the basis $\basis_j \in V$ at least one sample is contained in $\textrm{supp}(\basis_j)$. 
If this is not the case then the Gramian matrix is singular, because the discrete inner product of two functions is equal to zero when none of the samples is contained in the intersection of their supports.    
\end{remark}

\section{Numerical methods for adaptive (polynomial) approximation}
\noindent
The results presented in Theorems~\ref{thm:teo_union_bounds_mio} and \ref{thm:teo_union_bounds} hold for any nested sequence $(V_\iteraz)_\iteraz$ of general approximation spaces, in any dimension $d$. 
Two families of spaces that are suitable for approximation in arbitrary dimension $d$ 
are polynomial spaces and wavelet spaces.  
In this paper we confine to polynomial spaces. 
Even with this restriction, adaptive numerical methods in such a general context are still quite a large subject. Our focus in the present paper is on a more specific type of adaptive methods, 
in the spirit of orthogonal matching pursuit, and 
on the line of the greedy algorithms described in \cite{DT}.

The spaces $V_\iteraz$ can be adaptively chosen from one iteration to the other, as long as the sequence remains nested. 
Without additional information on the function that we would like to approximate, the infinite number of elements in the basis prevents the development of a concrete strategy for performing the adaptive selection. 
Such additional information is available in the form of decay of the coefficients, for example, for some PDEs with parametric or stochastic coefficients, whose solution is provably well-approximated by so-called downward closed polynomial spaces.  
See \cite{CD} and references therein for an introduction to the topic. The definition of downward closed polynomial spaces is postponed to \eqref{down_clos_def}. 
In the remaining of this section we assume that 
\begin{equation}
\substack{
\textrm{\normalsize the function $u$ can be well approximated by a nested sequence}\\
\textrm{\normalsize of downward closed polynomial approximation spaces.} 
}
\label{eq:assumption_seq_down_clos}
\end{equation}
As a relevant example that motivates our interest in the above setting, for PDEs with lognormal diffusion coefficients it was shown by the author in \cite[Lemma 2.4]{CM2017} that suitable polynomial spaces yielding provable convergence rates are actually downward closed.  

After \eqref{eq:assumption_seq_down_clos} we restrict our analysis to nested sequences $(V_\iteraz)_\iteraz$ of polynomial spaces satisfying the additional constraint of being downward closed. 
At iteration $\iteraz$, given $V_{\iteraz-1}$, an ideal (local) optimal criterion for performing the adaptive selection is to choose $V_{\iteraz}\supset V_{\iteraz-1}$ as the space that delivers the smallest error among all possible downward closed spaces with prescribed dimension, 
for example $\spazio_{\iteraz}=1+\spazio_{\iteraz-1}$. 
Since $d$ is finite, the number of all possible choices for $V_\iteraz$ is also finite.  
In reality the exact error $\|u-\Pi_{\spazio_\iteraz} u\|$ is not available, and the adaptive selection has to rely on the error $\|u-u_C^\iteraz \|$ that is a random variable. 
Here the error estimates from Theorems~\ref{thm:teo_union_bounds_mio} and \ref{thm:teo_union_bounds} come in handy because they ensure that $\|u-u_C^\iteraz\|^2$ is less than twice $\|u-\Pi_{\spazio_\iteraz} u\|^2$ in expectation. 
Even if the exact error was available, the adaptive selection using the local optimal criterion does not ensure optimality of the selected spaces at the following iterations, and for this reason it is referred to as a \emph{greedy} adaptive selection. 

Before moving to the description of the adaptive algorithm, we briefly introduce some definitions that are useful to describe the polynomial setting. 
Hereafter we assume that $\setx=\times_{i=1}^d I_i$ is the Cartesian product of intervals $I_i \subset \mathbb{R}$, and that $d\errmes=\otimes_{i=1}^d d\errmes_i$ where each $\errmes_i$ is a probability measure defined on $I_i$. 
This setting ensures the existence of a product basis orthonormal in $L^2(\setx,\errmes)$ that we now introduce.   
To simplify the presentation and notation, we further suppose that $I:=I_j$ and $\tilde\errmes:=\errmes_j$ for any $j$, 
and denote with $(T_j)_{j\geq 1}$ the univariate family of orthogonal polynomials, orthonormal in $L^2(I, \tilde\errmes)$. 
Let $\Lambda \subset \cF := \mathbb{N}_0^{d}$ be a multi-index set enumerated according to an ordering relation, for example the lexicographical ordering. 
Using $\Lambda$ we define 
\begin{equation}
\basis_\nu(\sample) := \prod_{i=1}^d T_{\nu_i}(\sample_i), \quad \nu=(\nu_1,\ldots,\nu_d) \in \Lambda,\quad \sample=(\sample_1,\ldots,\sample_d) \in \setx, 
\label{eq:multiv_basis}
\end{equation}
and relate the orthonormal basis $(\basis_i)_{i\geq 1}$ from the previous sections to the above orthonormal basis as $\basis_\indbasis = \basis_{\nu^\indbasis}$ for any $i=1,\ldots,\#(\Lambda)$, where $\nu^\indbasis$ is the $\indbasis$th element of $\Lambda$ according to the lexicographical ordering, and $\#(\Lambda)$ denotes the cardinality of $\Lambda$.  
The space associated to $\Lambda$ is defined as $V_{\Lambda}:=\textrm{span}\left\{ \basis_\nu : \nu \in \Lambda  \right\}$. 

A set $\Lambda \subset \cF$ is downward closed if 
\begin{equation}
\nu \in \Lambda \textrm{ and } \nu^\prime \leq \nu \implies \nu^\prime \in \Lambda, 
\label{down_clos_def}
\end{equation}
where the ordering $\nu^\prime \leq \nu$ is intended in the lexicographical sense. 
We say that the space $V_\Lambda$ is downward closed if the supporting index set $\Lambda$ is downward closed. 
For any $\Lambda \subset \cF$ downward closed we define its margin 
$\mathcal{M}(\Lambda)$ as 
$$
\mathcal{M}(\Lambda) := 
\left\{
\nu \in 
\cF 
: 
\nu \notin \Lambda \wedge \exists j 
\in \{1,\ldots,d\} : 
\nu  - e_j \in \Lambda
  \right\}, 
$$
where $e_j \in \cF
$ is the multi-index with all components equal to zero, except the $j$th component  that is equal to one. 
The reduced margin $\mathcal{R}(\Lambda)$ of $\Lambda$ is defined as 
$$
\mathcal{R}(\Lambda):=\{  \nu \in 
\cF
: \nu \notin \Lambda \wedge 
\forall j \in \{1,\ldots, d\}, \, \nu_j \neq 0 \implies \nu - e_j \in \Lambda
 \} \subseteq \mathcal{M}(\Lambda). 
$$
If $\Lambda$ is downward closed then $\Lambda \cup \{\nu\}$ is downward closed for any $\nu \in \mathcal{R}(\Lambda)$.  

Finally we choose the space $V_\iteraz$ from the previous sections as $V_\iteraz=V_{\Lambda_\iteraz}$ for any $\iteraz$ 
by means of a nested sequence $(\Lambda_\iteraz)_\iteraz \subset \cF$ of downward closed multi-index sets.   
For any $\iteraz \geq 1$,  $\#(\Lambda_\iteraz)=\textrm{dim}(V_\iteraz) = \spazio_\iteraz$ equals the dimension of $V_\iteraz$.

\subsection{An adaptive OMP algorithm}
\noindent 
In this section we describe an adaptive algorithm using optimal weighted least squares, starting from the algorithm proposed in \cite{M2014p} for standard least squares and inspired by orthogonal matching pursuit.  
The algorithm builds a sequence of nested spaces $V_{\Lambda_1} \subset \ldots \subset V_{\Lambda_\iterazmax}$ 
performing at each iteration an adaptive greedy selection of the indices identifying the elements of the basis. 
The adaptive construction of the index sets uses ideas that were originally proposed in \cite{GG2003} 
for developing adaptive sparse grids quadratures.  
The greedy selection of the indices uses a marking strategy known as bulk chasing. 

The adaptive algorithm works with downward closed index sets.
Given any $\Lambda$ downward closed, a  nonnegative function $e:\mathcal{R}(\Lambda) \to \mathbb{R}$ and a parameter $\paradorfler \in (0,1]$, we define the procedure $\textrm{BULK}:=\textrm{BULK}(\mathcal{R}(\Lambda),e,\paradorfler)$ that computes a set $F \subseteq \mathcal{R}(\Lambda) $ of minimal positive cardinality such that 
\begin{equation}
\label{eq:estimator_dorf}
\sum_{\nu \in F} e(\nu) \geq \paradorfler \sum_{\nu \in \mathcal{R}(\Lambda)} e(\nu). 
\end{equation}
Denote with $a_\nu$ the coefficient associated to $\basis_\nu$ in the expansion $u=\sum_{\nu \in 
\cF
} a_\nu \basis_\nu $. 
For any $\nu \in \mathcal{R}(\Lambda)$, the function $e(\nu)$ is chosen as an estimator for $|a_\nu|^2$.  
The adaptive algorithm is described in Algorithm~\ref{algo:adaptive}. 
At any iteration $\iteraz$ of the algorithm, $\punti_\iteraz$ random samples are generated by using Algorithm~\ref{algo:sampling_mio_simplif}, with 
$\punti_\iteraz$ satisfying \eqref{eq:condition_points_zeta_tau} as a function of $\spazio_\iteraz=\#(\Lambda_\iteraz)$, for a given choice of the parameters $\conf$ and $\paramzeta$. 
The $\punti_\iteraz$ random samples are used to compute the weighted least-squares estimator $u_C^{\iteraz}$ on $V_{\Lambda_\iteraz}$.  
For convenience in Algorithm~\ref{algo:adaptive} the operations performed by Algorithm~\ref{algo:sampling_mio_simplif} have been merged with those for the adaptive selection of the space. 
In Algorithm~\ref{algo:adaptive} the $\newmes_\nu$ correspond to $\newmes_j$ with $\basis_j=\basis_\nu$. 
An estimator for $|a_\nu|^2$ proposed in \cite{M2014p} 
that uses only the information available at iteration $\iteraz-1$ is 
\begin{equation}
e_{\iteraz-1}(\nu) := \left| \langle 
u-u_C^{\iteraz-1}, \basis_\nu \rangle_{\punti_{\iteraz-1}}  \right|^{2}, \quad \nu \in \mathcal{R}(\Lambda_{\iteraz-1}), 
\label{eq:estimator_coeff}
\end{equation}
where the discrete inner product uses the evaluations of the function $u$ at the same $\punti_{\iteraz-1}$ samples that have been used to compute $u_C^{\iteraz-1}$ at iteration $\iteraz-1$. 
The estimator \eqref{eq:estimator_coeff} uses 
the residual $r_{\iteraz-1}:=u-u_C^{\iteraz-1}$ 
and is cheap to compute: it requires only the product of a vector with a matrix. 

A safeguard mechanism prevents Algorithm~\ref{algo:adaptive} from getting stuck into indices associated to null coefficients in the expansion of $u_C^{\iteraz}$. 
Given a positive integer $\iteraz_{\textrm{sg}}$, once every $\iteraz_{\textrm{sg}}$ iterations the algorithm adds to $\Lambda_{\iteraz}$ the most ancient multi-index from $\mathcal{R}(\Lambda_{\iteraz-1})\setminus F$. 
In the numerical tests reported in the next section, such a mechanism was never activated, and the algorithm was always able to identify the best $\spazio_\iteraz$-term index sets of the given function at any iteration $\iteraz$.  

Algorithm~\ref{algo:adaptive} can be modified by relaxing \eqref{eq:condition_points_zeta_tau} to a less demanding condition between $\punti_\iteraz$ and $\spazio_\iteraz$ at each iteration $\iteraz$. 
For example, the random samples can be added until a stability condition of the form $\vvvert G_\iteraz - I_\iteraz \vvvert \leq \thresholdcond$ is met, for some given threshold $\thresholdcond > 1/2$. 
This provides a fully adaptive algorithm as described in Algorithm~\ref{algo:fullyadaptive}, that however, 
in contrast to Algorithm~\ref{algo:adaptive}, does not come with the theoretical guarantees of Theorem~\ref{thm:teo_union_bounds_mio}. 

\subsection{Testing the sampling algorithms}
\noindent
This section presents some numerical tests of the sampling algorithms that generate the random samples, comparing Algorithm~\ref{algo:sampling_mio_simplif} and Algorithm~\ref{algo:sampling_altro}. 
At the very end, our implementation of both algorithms uses inverse transform sampling as described in \cite[Section 5.2]{CM2016} 
for drawing samples from all the $\newmes_j$. 

A natural vehicle to quantify the quality of the generated samples is the deviation of the matrix $G_\iteraz$ from the identity, \emph{i.e.} $\vvvert G_\iteraz - I_\iteraz \vvvert$. 
Since $\vvvert G_\iteraz - I_\iteraz \vvvert \leq \frac12 \implies  \textrm{cond}(G_\iteraz):=\vvvert G_\iteraz^{-1} \vvvert \vvvert G_\iteraz \vvvert \leq 3$, our tests show the condition number, that is a more meaningful quantity when solving a linear system.   

From the point of view of the stability and convergence properties of the weighted least-squares estimators, the random samples generated  by both algorithms come with the same theoretical guarantees. 
But, in contrast to Algorithm~\ref{algo:sampling_altro}, Algorithm~\ref{algo:sampling_mio_simplif} recycles all the samples from the previous iterations. 
This is the main reason to prefer Algorithm~\ref{algo:sampling_mio_simplif} over Algorithm~\ref{algo:sampling_altro}. 
Another reason to choose Algorithm~\ref{algo:sampling_mio_simplif} is that it produces more stable Gramian matrices on average, since the sample variance of the generated samples is lower. 

\begin{figure}[htbp]
\begin{center}
\includegraphics[scale=0.4, bb=128 266 469 561]{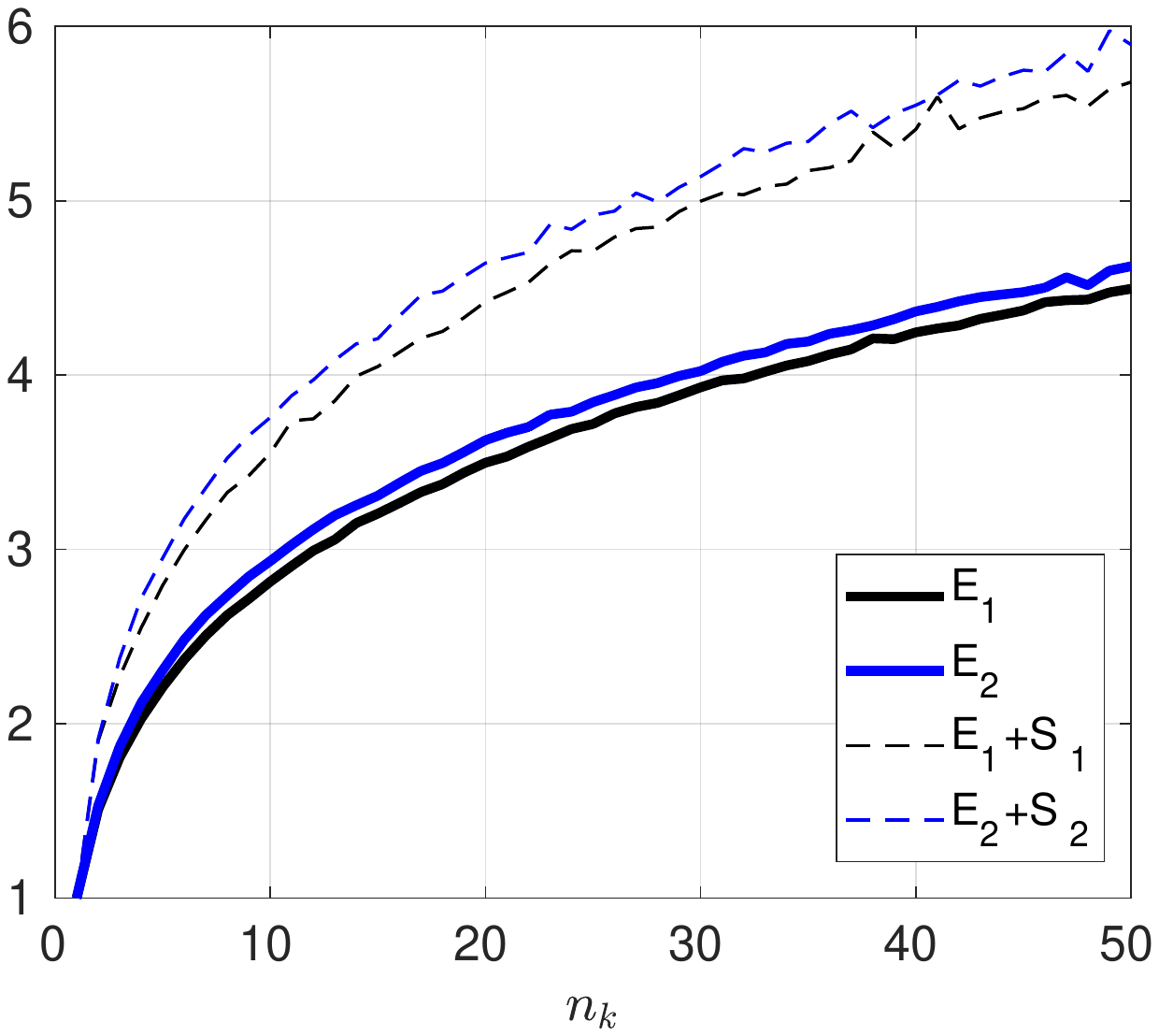}
\hspace{0.35cm}
\includegraphics[scale=0.43, bb=118 276 469 561]{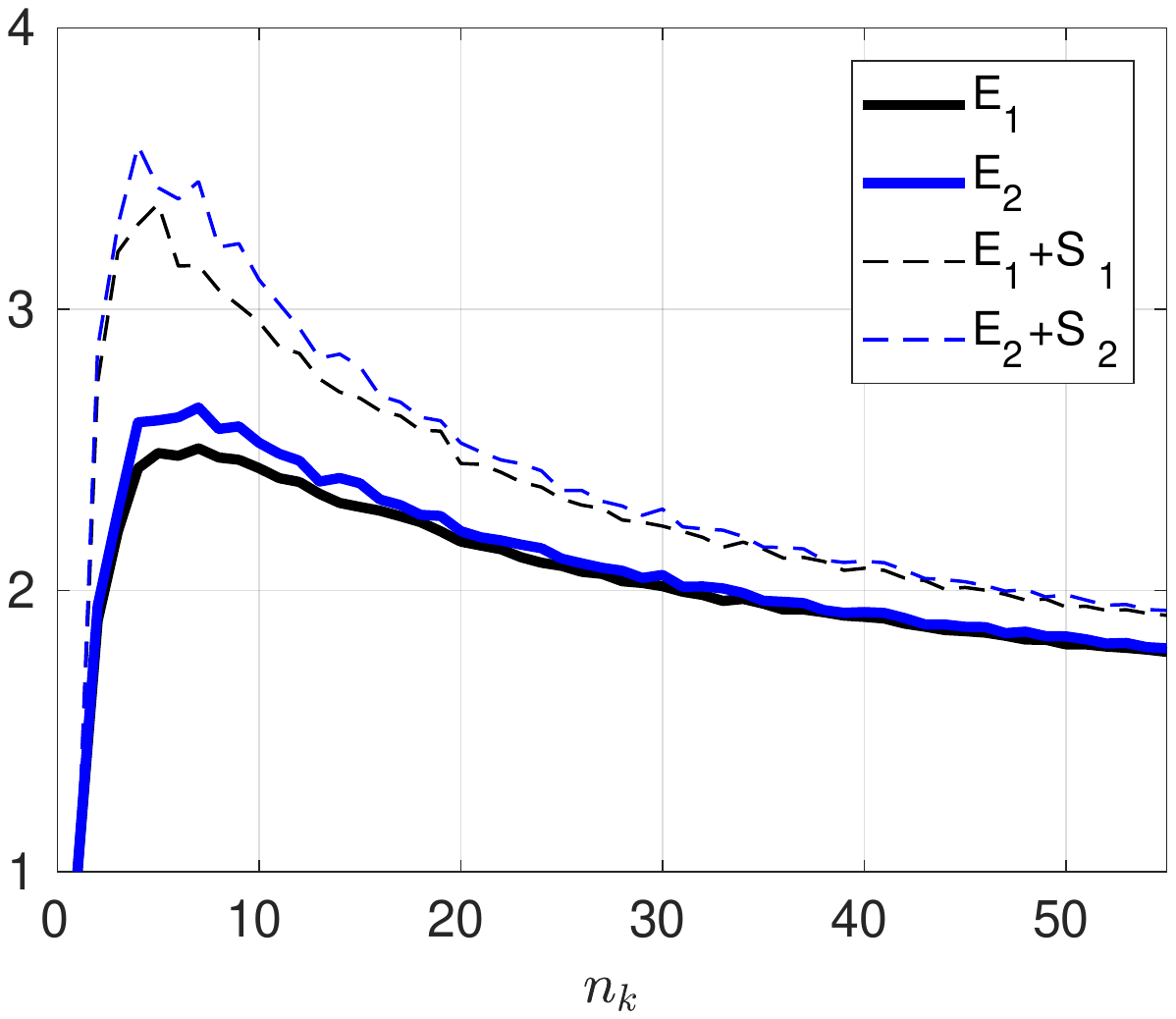}
\hspace{0.7cm}
\includegraphics[scale=0.4, bb=128 266 469 561]{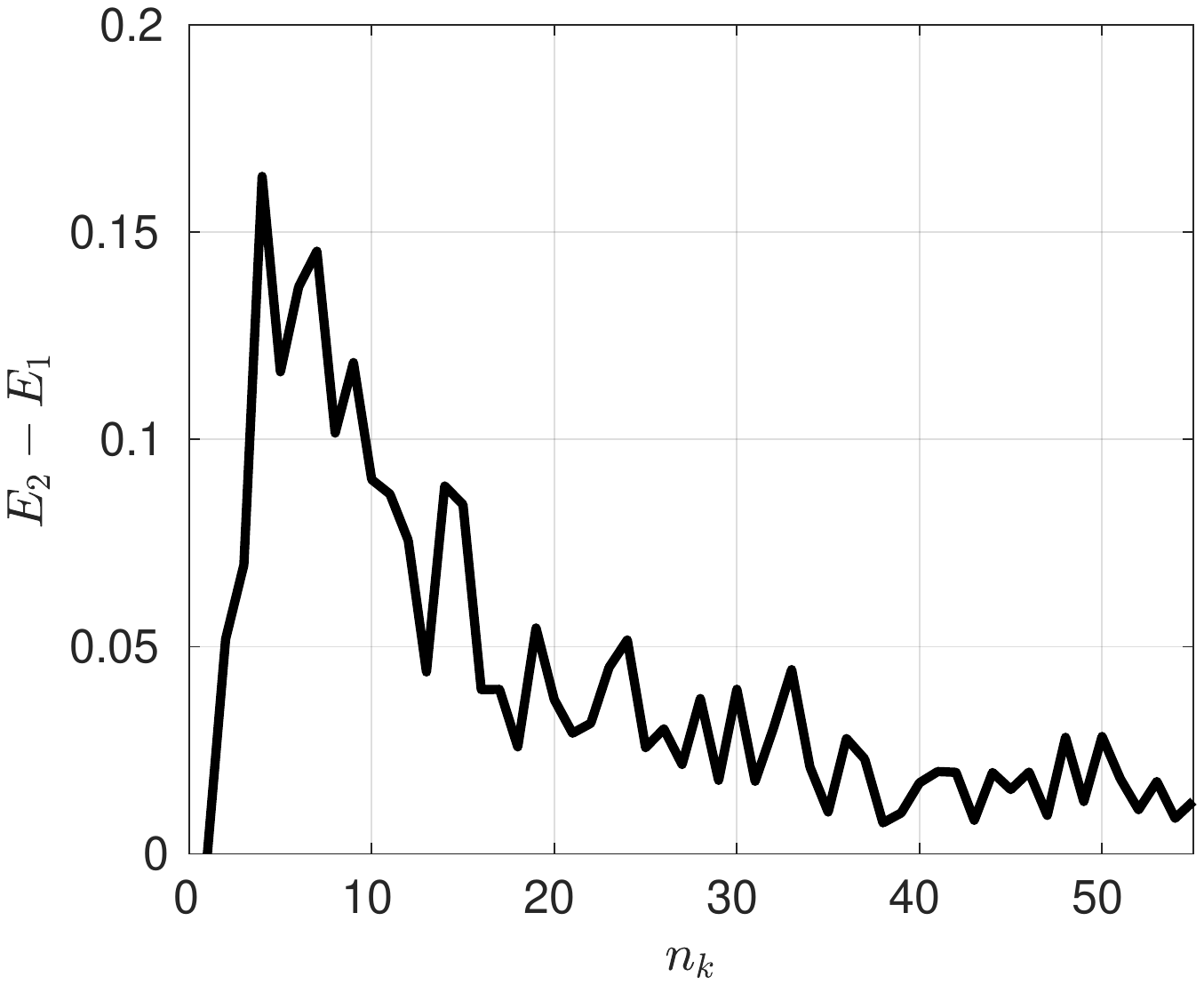}
\end{center}
\vspace{-0.6cm}
\caption{
Left: estimators $E_i$ and $E_i+S_i$ of the sequence of random variables $(\textrm{cond}(G_\iteraz))_{\iteraz\geq 1}$ at iteration $\iteraz=1,\ldots,50$ with $\spazio_\iteraz=\iteraz$ and $\punti_\iteraz= \lceil \costprop^{-1} \rceil \spazio_\iteraz $.  Hermite polynomials. $d=1$.  
The estimators use $10^4$ realizations of the sequence $(\textrm{cond}(G_\iteraz))_{\iteraz\geq 1}$.
Center: estimators $E_i$ and $E_i+S_i$ of the sequence of random variables $(\textrm{cond}(G_\iteraz))_{\iteraz\geq 1}$  
at iteration $\iteraz=1,\ldots,55$ with $\spazio_\iteraz=\iteraz$ and 
$\punti_\iteraz= (3+\spazio_\iteraz) \spazio_\iteraz $.  Hermite polynomials. $d=1$.  
The estimators use $10^3$ realizations of the sequence $(\textrm{cond}(G_\iteraz))_{\iteraz\geq 1}$.
Right: same simulation as center but showing $E_2-E_1$. 
}
\label{fig:algos_samp_comp}
\end{figure}

Our first tests illustrate the benefits of variance reduction, using spaces $V_\iteraz$ of univariate Hermite polynomials $(H_j)_{j\geq 0}$ with degree from $0$ to $\iteraz-1$.  
More precisely, the sequence $(H_j)_{j\geq 0}$ contains univariate Hermite polynomials orthonormalised as $\int_{\mathbb{R}} H_i(t) H_j(t) \, dg = \delta_{ij}$, where $dg:=(2\pi)^{-1/2} e^{-t^2/2} \, dt$. 
Denote with $E_{i}\approx\mathbb{E}(\textrm{cond}(G_\iteraz))$ and $S_{i}^2\approx \textrm{Var}(\textrm{cond}(G_\iteraz))$ the sample mean and sample variance estimators of the random variable $\textrm{cond}(G_\iteraz)$ 
with $G_\iteraz$ constructed using the random samples generated by Algorithm $i \in \{1,2\}$. 
Figure~\ref{fig:algos_samp_comp}-left shows the comparison of $E_i$ and $E_i+S_i$ between the two algorithms, with $\punti_\iteraz=\lceil \costprop^{-1} \rceil \spazio_\iteraz$. 
Both estimators confirm that Algorithm~\ref{algo:sampling_mio_simplif} produces random samples whose Gramian matrix is better conditioned than Algorithm~\ref{algo:sampling_altro}. 
The same trend persists when choosing other scalings like $\punti_\iteraz=(3+\spazio_\iteraz) \spazio_\iteraz$, see Figure~\ref{fig:algos_samp_comp}-center and Figure~\ref{fig:algos_samp_comp}-right.    
The difference between the two algorithms is expected to amplify when using more localized basis, with Algorithm~\ref{algo:sampling_altro} producing much more ill-conditioned Gramian matrices as the ratio $\punti_\iteraz/\spazio_\iteraz$ decreases.

From now on the focus is on Algorithm~\ref{algo:sampling_mio_simplif}. For all the tests in the remaining part of this section we choose $\punti_\iteraz$ as in \eqref{eq:condition_points_zeta_tau} with $\conf=0.1$ and $\paramzeta=2$. The value of $\conf$ is chosen fairly large on purpose to check, in practice, how sharp the stability constraint \eqref{eq:condition_points_zeta_tau} is. 
In the first test, we choose $\errmes=\otimes^d dg$ as the $d$-dimensional probabilistic Gaussian measure on $\setx=\mathbb{R}^d$, and $V_\iteraz$ as the spaces of tensorized Hermite polynomials, obtained from \eqref{eq:multiv_basis} by taking $T_j=H_j$, ${j\geq 0}$. 
The Gaussian case poses several challenges: as shown in \cite{CM2016}, standard least-squares estimators with Hermite polynomials typically fail due to the ill-conditioning of the Gramian matrix. 
Since the ill-conditioning arises with high-degree polynomials, we choose fairly low-dimensional tests to begin with, such that very high-degree polynomials can be tested, \emph{e.g.}~degrees beyond 100. 
With $d=1$ the results are shown in Figure~\ref{fig:cond_hermite}-left, 
and with $d=4$ in Figure~\ref{fig:cond_hermite}-right.    
The condition number of $G_\iteraz$ stays well below the threshold equal to $3$ during all the simulations,   
which contain, respectively, $10^4$ and $10^3$ realizations of the sequence $(\textrm{cond}(G_\iteraz))_{\iteraz\geq 1}$ with random samples generated by  Algorithm~\ref{algo:sampling_mio_simplif}. 
At any iteration $\iteraz$, the index set $\Lambda_{\iteraz}\supset \Lambda_{\iteraz-1}$ that defines the space $V_{\iteraz} =V_{\Lambda_\iteraz}$ is generated by adding to $\Lambda_{\iteraz-1}$ a random number of indices randomly chosen from $\mathcal{R}(\Lambda_{\iteraz-1})$. 
This procedure generates nested sequences of downward closed index sets, see Figure~\ref{fig:legendre_and_index_set}-right for an example of such a set.  
With other families of orthogonal polynomials the results are very similar. 
For example, with $d=4$, the results in Figure~\ref{fig:legendre_and_index_set}-left with the $d$-dimensional uniform probabilistic measure on $\setx=[-1,1]^d$ and Legendre polynomials are analogous to those obtained in Figure~\ref{fig:cond_hermite}-right with the Gaussian measure and Hermite polynomials.
Figure~\ref{fig:legendre_and_index_set}-right shows an example of (the section of the first and second coordinates of) an index set $\Lambda_{\iteraz}$ obtained in the simulation of Figure~\ref{fig:cond_hermite}-right at iteration $\iteraz=500$.
This set contains products of univariate Hermite polynomials with degree over $110$ in the first coordinate and up to $59$ in the second coordinate, and degree up to $25$ and $9$ in the remaining third and fourth coordinates not displayed in the figure.  
\begin{figure}[htbp]
\begin{center}
\includegraphics[scale=0.4, bb=83 252 512 590]{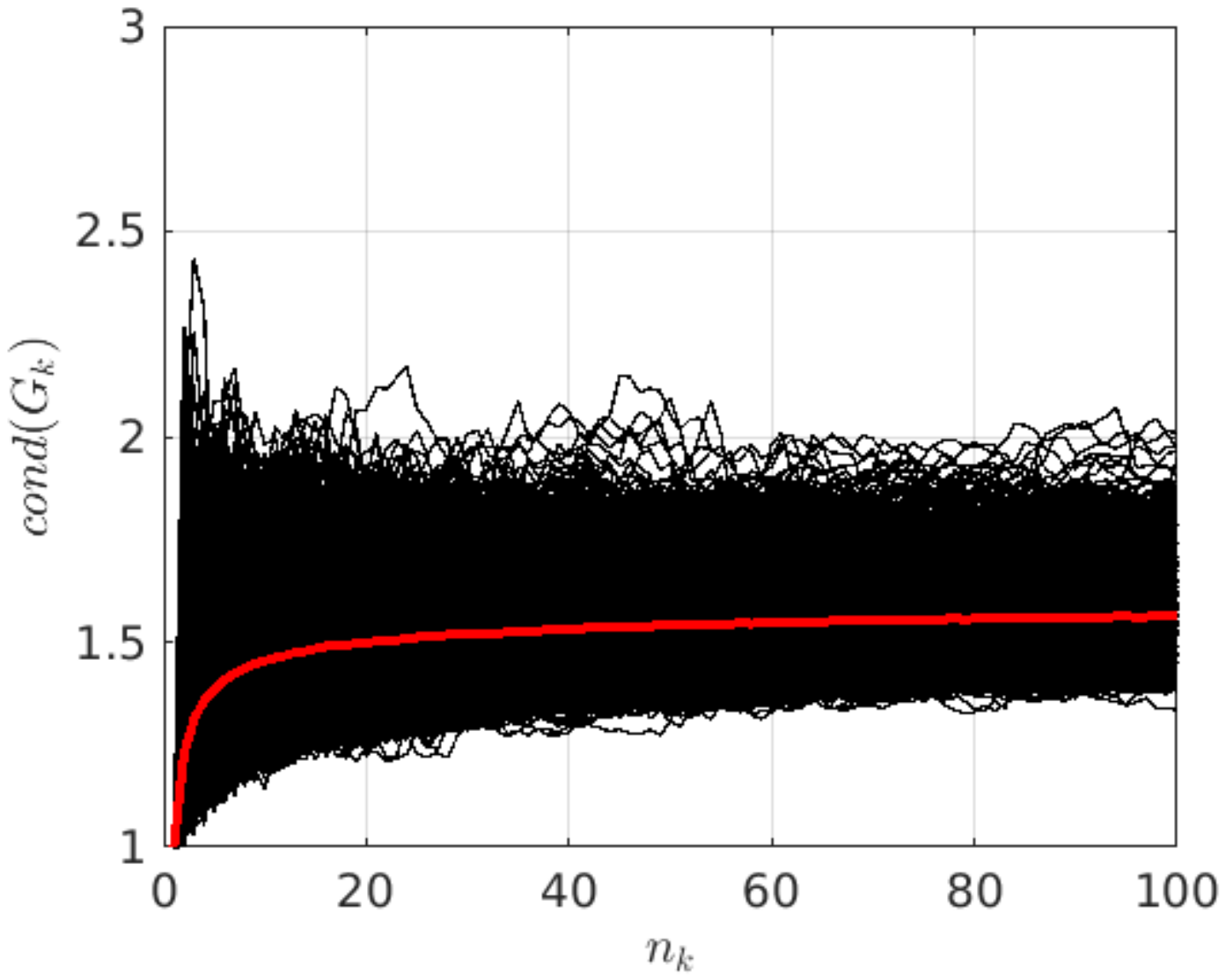}
\hspace{1.0cm}
\includegraphics[scale=0.4, bb=89 258 481 570]{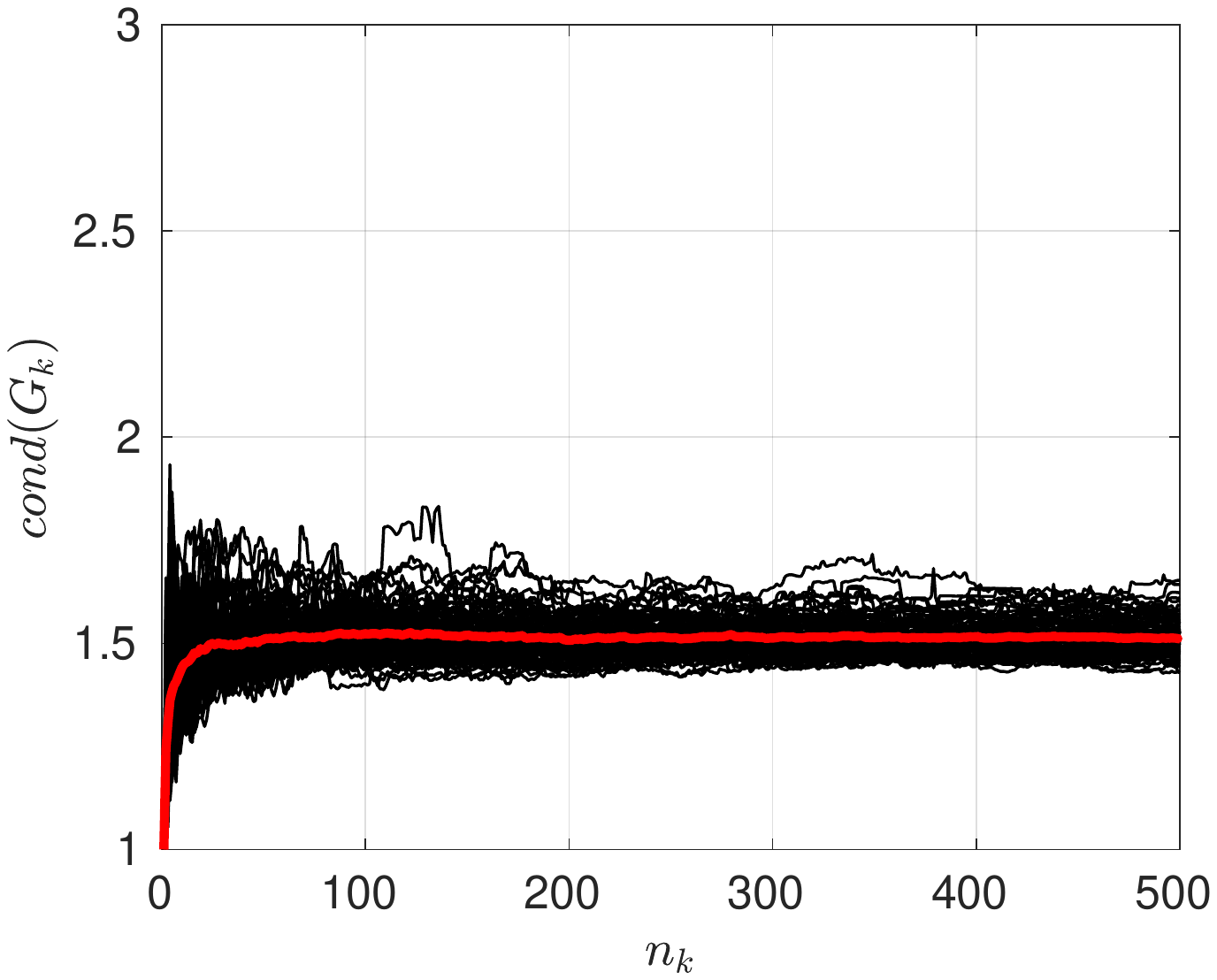}
\end{center}
\vspace{-0.7cm}
\caption{
Left: condition number $\textrm{cond}(G_\iteraz)$ at iteration $\iteraz$ with $\spazio_\iteraz=\iteraz$ and 
$\punti_\iteraz$ as in \eqref{eq:cond_tau_points}, $d=1$, Gaussian measure, Hermite polynomials,   $\paramzeta=2$, $\conf=0.1$. 
Black lines are $10^4$ realizations of the sequence $(\textrm{cond}(G_\iteraz))_{\iteraz\geq 1}$ with random samples from Algorithm~\ref{algo:sampling_mio_simplif}.
The red line is their sample mean. 
Right: condition number $\textrm{cond}(G_\iteraz)$ at iteration $\iteraz$ with $\spazio_\iteraz=\iteraz$
and $\punti_\iteraz$ as in \eqref{eq:cond_tau_points}, $d=4$, Gaussian measure, Hermite polynomials, $\paramzeta=2$, $\conf=0.1$. 
Black lines are $10^3$ realizations of the sequence $(\textrm{cond}(G_\iteraz))_{\iteraz\geq 1}$ with random samples from Algorithm~\ref{algo:sampling_mio_simplif}.
The red line is their sample mean. 
}
\label{fig:cond_hermite}
\end{figure}
\begin{figure}[htbp]
\begin{center}
\includegraphics[scale=0.4, bb=89 258 481 570]{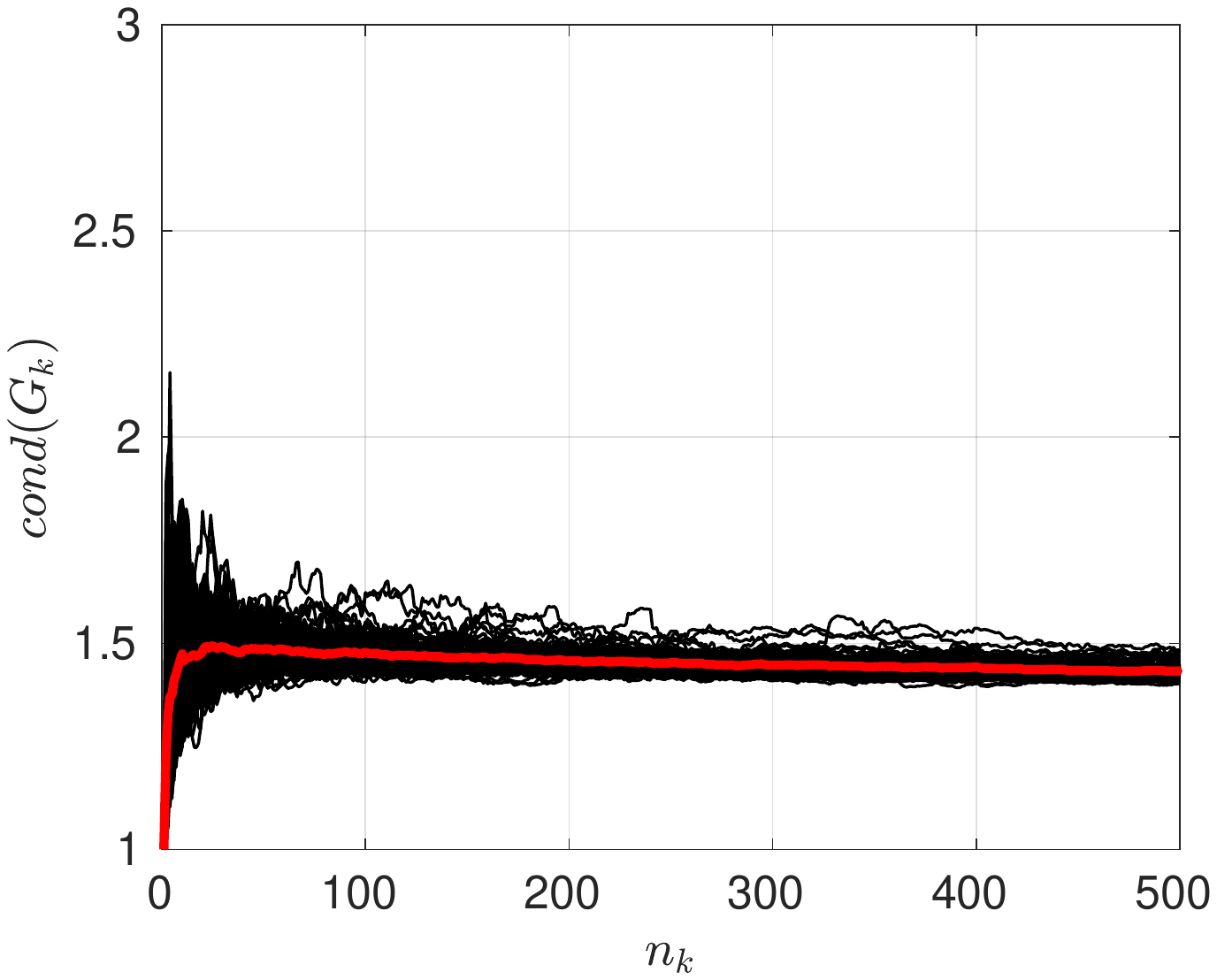}
\hspace{1.0cm}
\includegraphics[scale=0.4, bb=89 258 481 570]{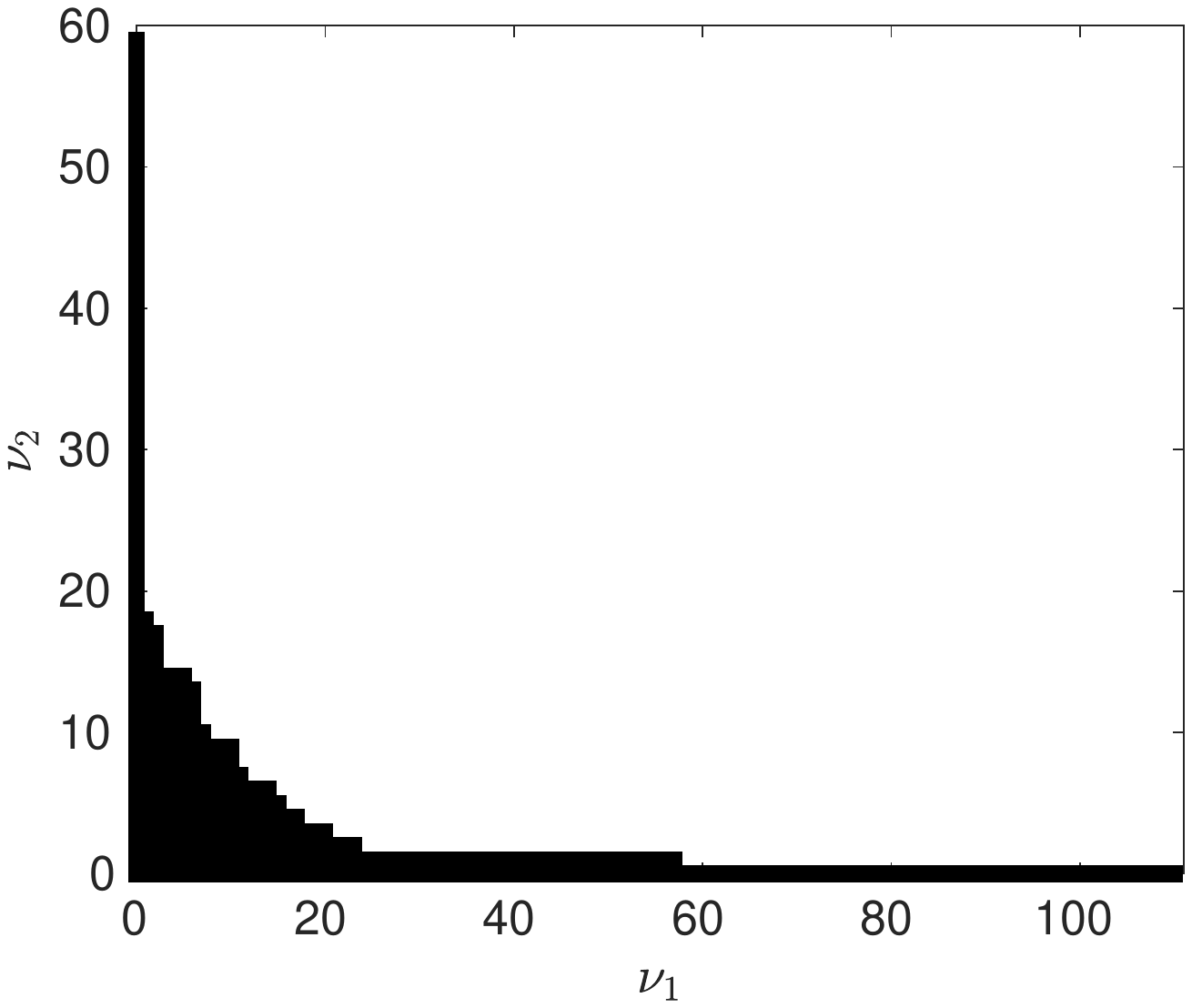}
\end{center}
\vspace{-0.7cm}
\caption{
Left: condition number $\textrm{cond}(G_\iteraz)$ at iteration $\iteraz$ with $\spazio_\iteraz=\iteraz$ and $\punti_\iteraz$ as in \eqref{eq:cond_tau_points}, $d=4$, uniform measure, Legendre polynomials, $\paramzeta=2$, $\conf=0.1$. 
Black lines are $10^3$ realizations of the sequence $(\textrm{cond}(G_\iteraz))_{\iteraz\geq 1}$ with random samples from Algorithm~\ref{algo:sampling_mio_simplif}.
The red line is their sample mean. 
Right: section of the first and second coordinates of an index set obtained at iteration $\iteraz=500$ during 
the simulation in Fig~\ref{fig:cond_hermite}-right. 
}
\label{fig:legendre_and_index_set}
\end{figure}

\subsection{Testing the adaptive algorithm}
\noindent 
For the numerical tests of Algorithm~\ref{algo:adaptive} we choose $\errmes$ as the uniform measure over $\setx=[-1,1]^d$ and $V_\iteraz$ as the spaces of tensorized Legendre polynomials 
obtained by first defining the sequence $(L_j)_{j\geq 0}$ of univariate Legendre polynomials orthonormalised as $\int_{-1}^{+1} L_i(t) L_j(t) \, \frac{dt}{2} = \delta_{ij}$ and then taking $T_j=L_j$ in \eqref{eq:multiv_basis}.
As an illustrative example, consider the following function that satisfies assumption \eqref{eq:assumption_seq_down_clos},  
\begin{equation}
u(\sample)= 
\left(
1 + 
\frac{1}{2d}
\sum_{i=1}^d q_i \sample_i 
\right)^{-1}, \qquad \sample \in \setx,
\label{eq:lafuncnum}
\end{equation}
with $d=16$ and $q_i = 10^{ -\frac{ 3(i-1) }{d-1} }$. 
A set $X_{CV}$ of $10^6$ cross-validation points uniformly distributed over $\setx$ is chosen once and for all, and the approximation error $\| u- u_C^\iteraz\|$ is estimated as   
\begin{equation}
\label{eq:error_cv}
\| u - u^\iteraz_C \| \approx
\| u - u^\iteraz_C \|_{CV,2}:= \sqrt{ \frac{1}{\#(X_{CV})  } \sum_{ \tilde \sample \in X_{CV} } | u(\tilde \sample) - u_C^\iteraz(\tilde \sample) |^2 
}
\leq  
\| u - u^\iteraz_C \|_{CV,\infty} :=
\max_{\tilde \sample \in X_{CV}} | u(\tilde \sample) - u_C^\iteraz(\tilde \sample) |. 
\end{equation}
The error estimators are denoted with $\| u - u^\iteraz_C \|_{CV,2}, \| u - u^\iteraz_C \|_{CV,\infty}$,  although these are not norms over the functional space.
The parameter of the marking strategy is set to $\paradorfler=0.5$, and $\Lambda_1=\{(0,\ldots,0)^\top \}.$ 
Figure~\ref{fig:adaptive_leg_err_cond}-left shows the results for the errors \eqref{eq:error_cv} obtained when approximating the function \eqref{eq:lafuncnum} with Algorithm~\ref{algo:adaptive} and using the random samples generated by Algorithm~\ref{algo:sampling_mio_simplif}. 
At each iteration $\iteraz$ the number of samples $\punti_\iteraz$ as a function of $\spazio_\iteraz$ satisfies \eqref{eq:condition_points_zeta_tau} with $\conf=0.1$ and $\paramzeta=2$. 
Figure~\ref{fig:adaptive_leg_err_cond}-right shows the condition number of $G_\iteraz$ at iteration $\iteraz$, that stays below two at all the iterations. 
Figure~\ref{fig:adaptive_leg_coeff_ind}-left shows that at each iteration $\iteraz$ the adaptive algorithm catches the coefficients in the best $\spazio_\iteraz$-term set. The coefficients in Figure~\ref{fig:adaptive_leg_coeff_ind}-left have not been sorted, and they appear in the same order in which their corresponding elements of the basis were included in the approximation space by the adaptive selection procedure. 
After $35$ iterations the algorithm has adaptively constructed a sequence $\Lambda_1,\ldots,\Lambda_{35}$ of index sets. 
The set $\Lambda_{35}$ contains about $10^3$ indices, and its associated space $V_{\Lambda_{35}}$ provides an approximation error of the order $10^{-7}$ on average.  
Figure~\ref{fig:adaptive_leg_coeff_ind} shows some sections of $\Lambda_{35}$. 
All the $d$ coordinates in $\Lambda_{35}$ are active, \emph{i.e.} $\forall i\in \{1,\ldots,d\}, \exists \nu \in \Lambda : \nu_j > 0$. 

The condition number in Figure~\ref{fig:adaptive_leg_err_cond}-right actually decreases w.r.t.~$\iteraz$, showing that condition \eqref{eq:condition_points_zeta_tau} could be relaxed while still preserving the stability of the discrete projection, and yielding faster convergence rates w.r.t.~$\punti_\iteraz$ than those in Figure~\ref{fig:adaptive_leg_err_cond}-left. 

\begin{figure}[htbp]
\begin{center}
\includegraphics[scale=0.4, bb=46 227 528 571]{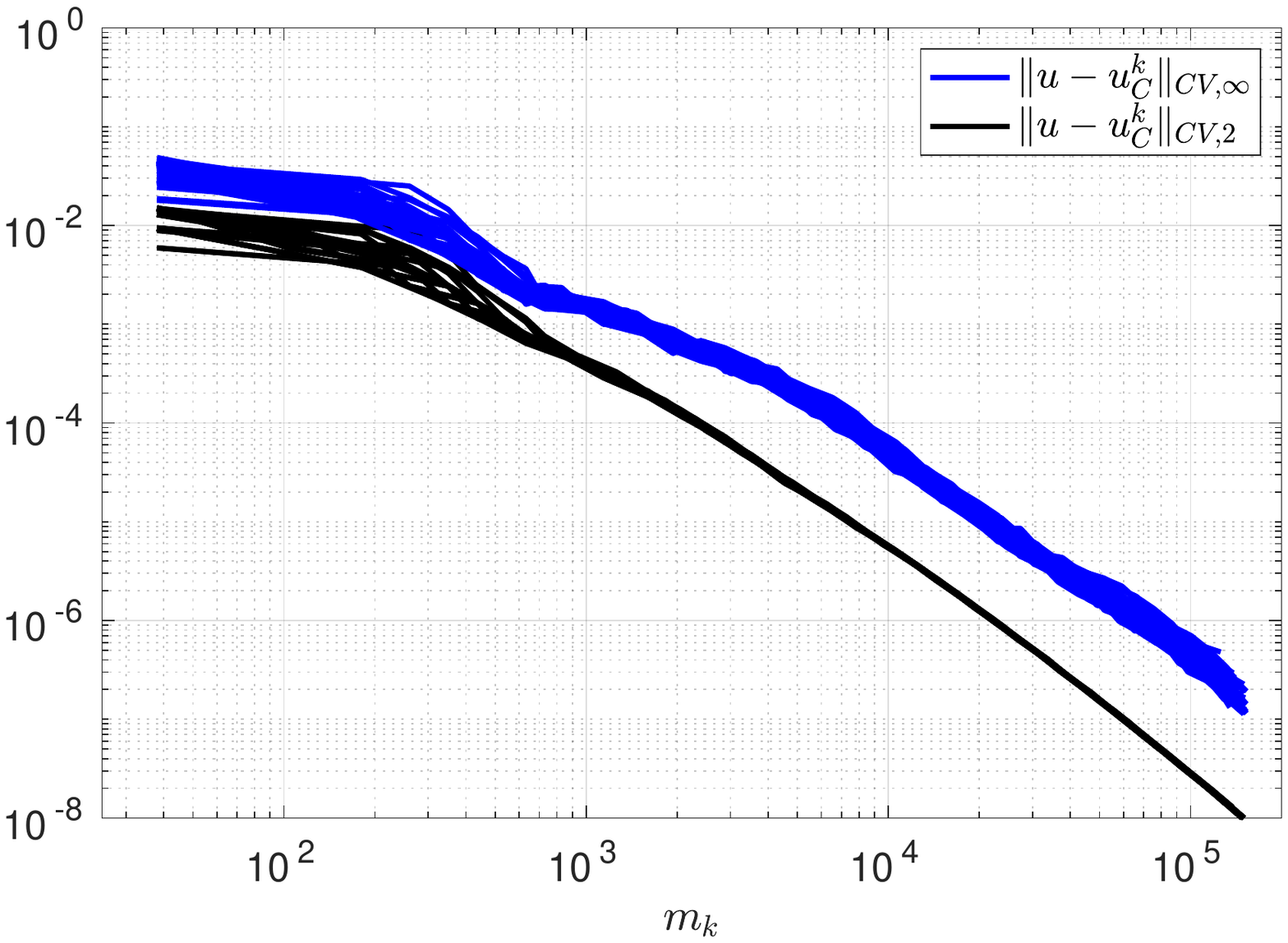}
\hspace{0.5cm}
\includegraphics[scale=0.4, bb=89 228 481 570]{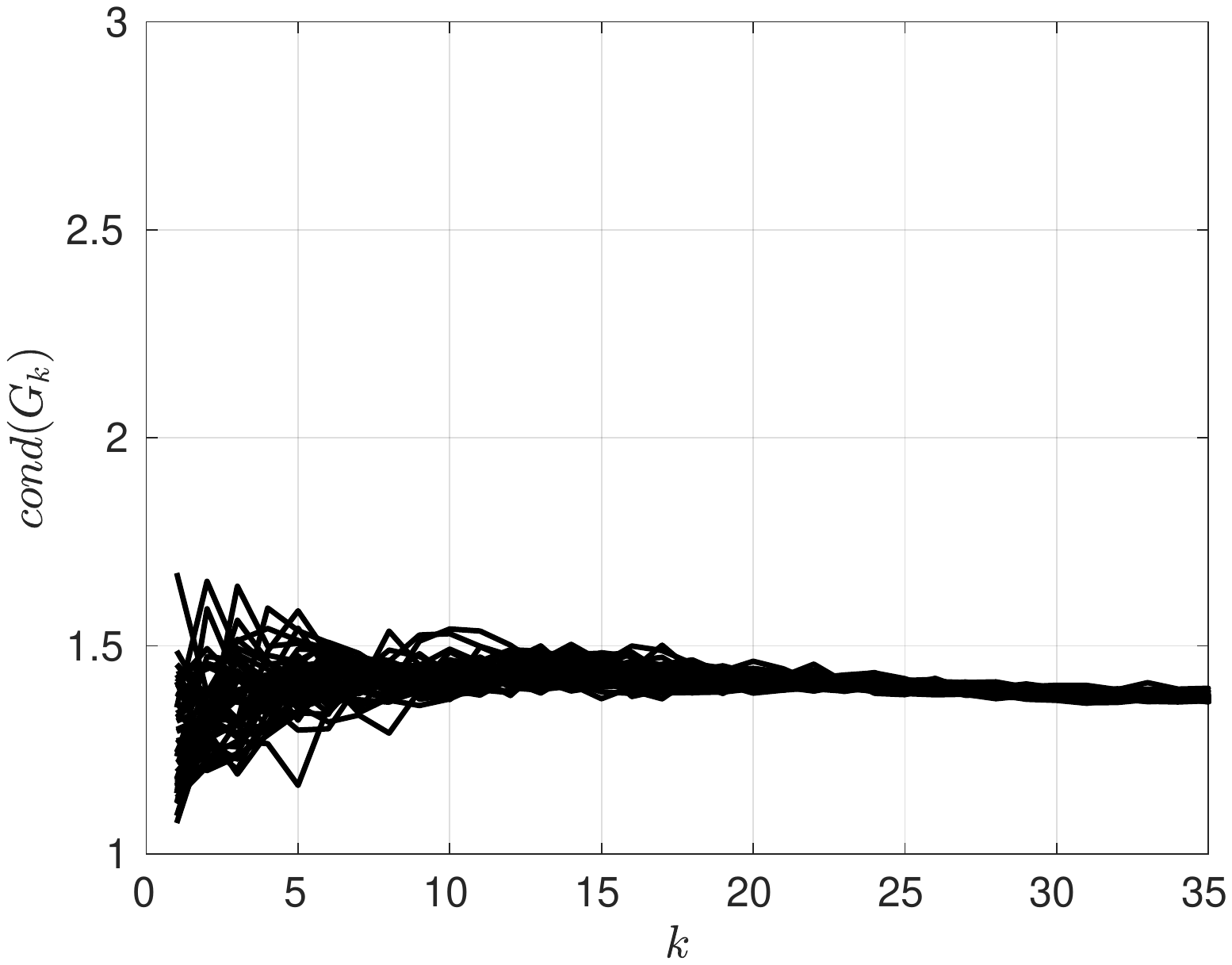}
\end{center}
\vspace{-0.5cm}
\caption{
Left: 
$10^2$ realizations of the errors 
$\| u - u^\iteraz_C \|_{CV,2}$ and 
$\| u - u^\iteraz_C \|_{CV,\infty}$ 
versus $\punti_\iteraz$ 
obtained with Algorithm~\ref{algo:adaptive} and the random samples generated by Algorithm~\ref{algo:sampling_mio_simplif}. 
Right: $10^2$ realizations of $\textrm{cond}(G_\iteraz)$ versus $\iteraz$, for the same simulation on the left. 
}
\label{fig:adaptive_leg_err_cond}
\end{figure}

\begin{figure}[htbp]
\begin{center}
\includegraphics[scale=0.4, bb=89 258 481 570]{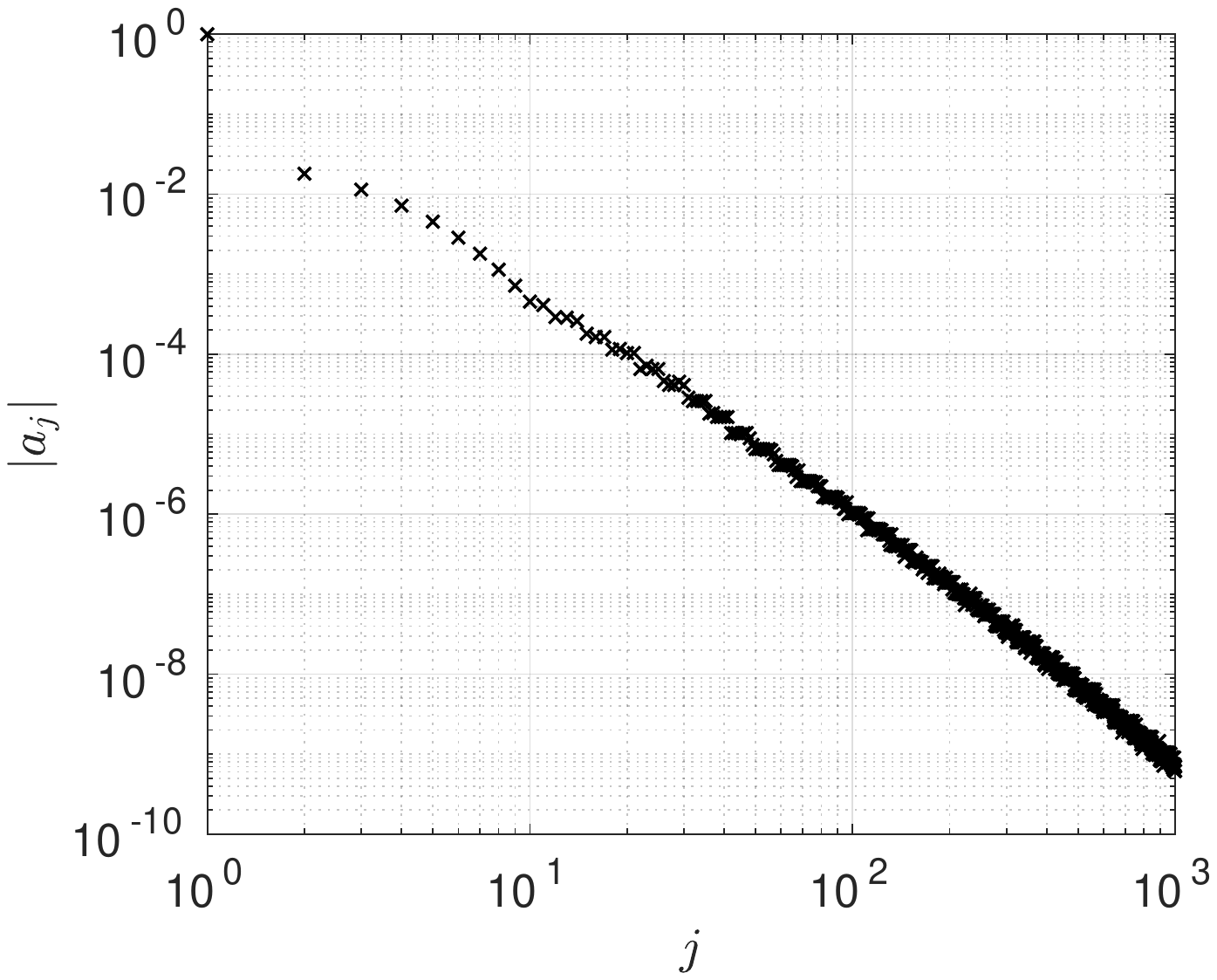}
\hspace{0.5cm}
\includegraphics[scale=0.4, bb=89 258 481 570]{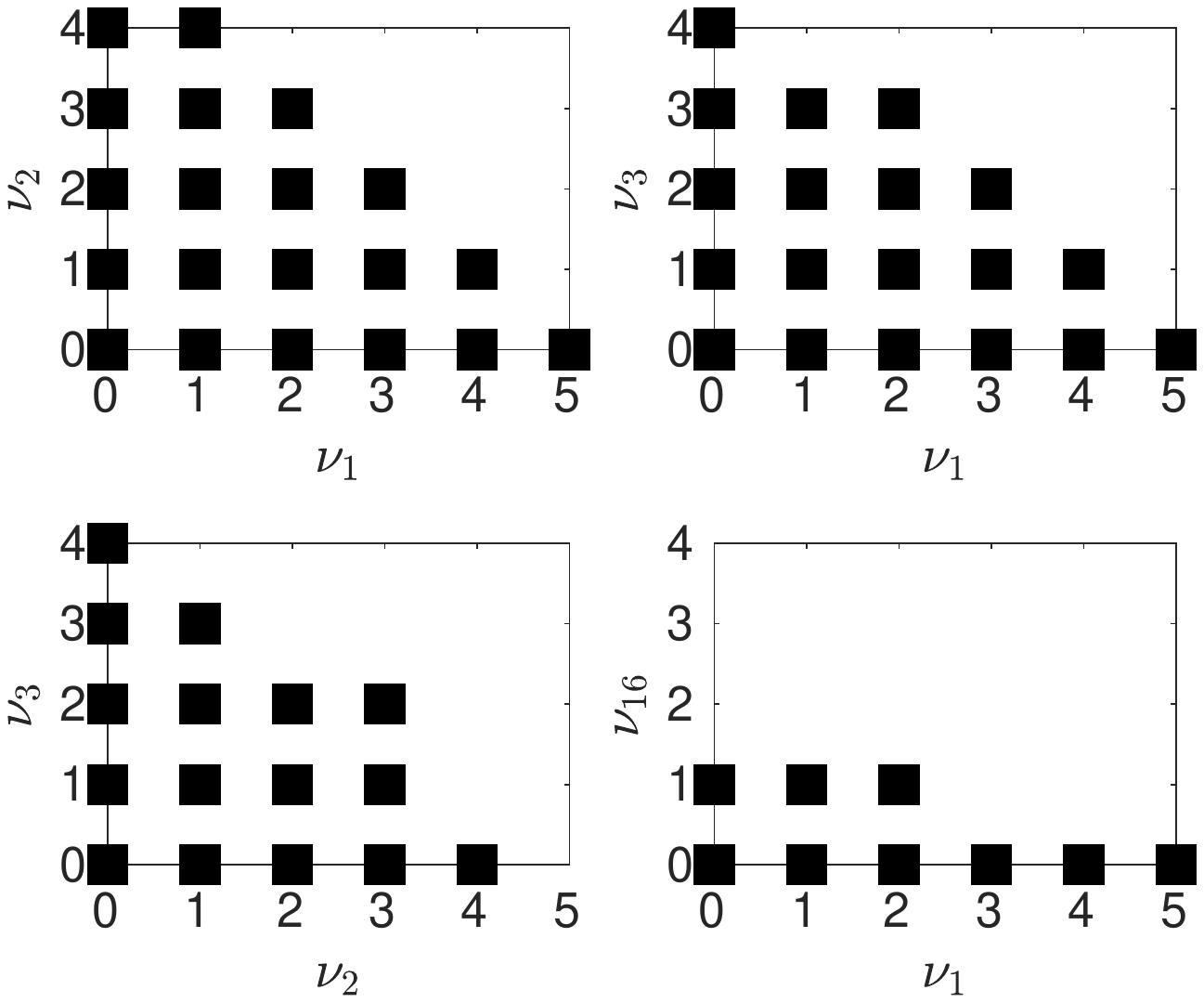}
\end{center}
\vspace{-0.5cm}
\caption{Left: first $10^3$ coefficients of the estimator 
$u_C^\iteraz=\sum_{j} a_j \basis_j$ 
obtained at iteration $\iteraz=35$ with the index set $\Lambda_{35}$, 
for one realization among those shown in Figure~\ref{fig:adaptive_leg_err_cond}. Right: some sections of the index set corresponding to the coefficients displayed on the left.}
\label{fig:adaptive_leg_coeff_ind}
\end{figure}

\section{Conclusions}
\noindent 
We have advanced one step further the analysis of optimal weighted least-squares estimators for a given general $d$-dimensional approximation space.   
The main novelty concerns the structure of the random samples, that follow a distribution with product form.  
The results have immediate applications to the adaptive setting with a nested sequence of approximation spaces, and point out new promising directions for the development of adaptive numerical methods for high-dimensional approximation using polynomial or wavelet spaces. 
Our analysis indicates that efficient adaptive methods can also be developed for general sequences of nonnecessarily nested spaces. This topic will be investigated in the future.

\bibliographystyle{plain}

\appendix

\section{Algorithms}

\begin{figure}[htbp]
\centering
\begin{algorithm}[H]\small
\caption{\small Deterministic sequential sampling}
\begin{algorithmic}
\REQUIRE $\iterazmax$, $d\errmes$, $(\intero_\iteraz)_{ \iteraz=1}^{\iterazmax}$, $(\spazio_\iteraz)_{ \iteraz=1}^{\iterazmax}$,  $(\basis_j)_{j=1}^{\spazio_\iterazmax}$  
\ENSURE $\sample^1,\ldots,\sample^{\punti_{\iterazmax} }$ 
s.t. 
$\sample^{(j-1)\intero_\iteraz +1}, \ldots, \sample^{j \intero_\iteraz}
\stackrel[]{\textrm{i.i.d.}}{\sim}
\newmes_j, 
j=1,\ldots,\spazio_\iteraz, k=1,\ldots,\iterazmax.
$ 
\FOR{$\indsample=1$ to $\spazio_{1}$}
\FOR{$\ell=1$ to $\intero_{1}$}
\STATE Sample $\sample^{\indsample \ell}$ from $\newmes_\indsample$
\ENDFOR
\ENDFOR
\STATE $(\sample^1,\ldots,\sample^{\punti_1})^\top \gets \textrm{Vec}\left( \left( \left(\sample^{\indsample \ell} \right)_{\indsample=1,\ldots,\spazio_1 \atop \ell=1,\ldots,\intero_1} \right)^\top  \right)  $
\FOR{$\iteraz=2$ to $\iterazmax$}
\FOR{$\indsample=\spazio_{\iteraz-1} +1$ to $\spazio_{\iteraz}  $}
\FOR{$\ell=1$ to $\intero_{\iteraz-1} $}
\STATE Sample $\sample^{\indsample \ell}$ from $\newmes_\indsample$
\ENDFOR
\ENDFOR
\FOR{$\indsample=1$ to $\spazio_{\iteraz}  $}
\FOR{$\ell= \intero_{\iteraz -1}+1$ to $\intero_{\iteraz} $}
\STATE Sample $\sample^{\indsample \ell}$ from $\newmes_\indsample$
\ENDFOR
\ENDFOR
\STATE $(\sample^1,\ldots,\sample^{\punti_\iteraz})^\top \gets \textrm{Vec}\left( \left( \left(\sample^{\indsample \ell} \right)_{\indsample=1,\ldots,\spazio_\iteraz \atop \ell=1,\ldots,\intero_\iteraz} \right)^\top  \right)  $
\ENDFOR
\end{algorithmic}
\label{algo:sampling_mio_simplif}
\end{algorithm}
\end{figure}

\begin{figure}[htbp]
\centering
\begin{algorithm}[H]\small
\caption{\small Random sequential sampling}
\begin{algorithmic}
\REQUIRE $\iterazmax$, $(\auxmes_{\spazio_\iteraz})_{\iteraz=1}^{\iterazmax-1}$, $(\sigma_{\spazio_\iteraz})_{\iteraz=1}^\iterazmax$ 
\ENSURE 
$\sample^1,\ldots,\sample^{\punti_{\iterazmax} }$ s.t. $\sample^1,\ldots,\sample^{\punti_\iteraz} \stackrel[]{\textrm{i.i.d.}}{\sim} \auxmes_{\spazio_\iteraz}, \iteraz=1,\ldots,\iterazmax$.
\FOR{$\indsample=1$ to $\punti_{1}$}
\STATE Sample $\sample^\indsample$ from $\auxmes_{\spazio_1}=\sigma_{\spazio_1}$
\ENDFOR
\FOR{$\iteraz=2$ to $\iterazmax$}
\STATE Sample $\binrv_\iteraz$ from $\Bin\left(\punti_{\iteraz}, \dfrac{\spazio_{\iteraz}-\spazio_{\iteraz-1}}{\spazio_\iteraz} \right)$
\FOR{$\indsample=\min\left( \punti_{\iteraz} - \binrv_\iteraz , \punti_{\iteraz-1}  \right)+1$ to $\min\left( \punti_{\iteraz} - \binrv_\iteraz , \punti_{\iteraz-1}  \right)+\max\left( \punti_\iteraz - \binrv_\iteraz - \punti_{\iteraz-1},0 \right)$}
\STATE Sample $\sample^\indsample$ from $\auxmes_{\spazio_{\iteraz-1}}$
\ENDFOR
\FOR{$\indsample=\min\left( \punti_{\iteraz} - \binrv_\iteraz , \punti_{\iteraz-1}  \right)+\max\left( \punti_\iteraz - \binrv_\iteraz - \punti_{\iteraz-1},0 \right)+1$ to $\punti_\iteraz$}
\STATE Sample $\sample^\indsample$ from $\sigma_{\spazio_\iteraz}$
\ENDFOR
\ENDFOR
\end{algorithmic}
\label{algo:sampling_altro}
\end{algorithm}
\end{figure}

\begin{figure}[htbp]
\centering
\begin{algorithm}[H]\small
\caption{\small Adaptive weighted least squares}
\begin{algorithmic}
\REQUIRE $\Lambda_1=\{ (0,\ldots,0)^\top \}$, 
$\paradorfler$, 
$\paramzeta$, $\conf$, 
$\iterazmax$, 
$\iteraz_{\textrm{sg}}$ 
\ENSURE $u_C^{\iterazmax}$
\STATE $\intero_1= \lceil \costprop^{-1} \ln( \zeta(\paramzeta) (\#(\Lambda_1))^{\paramzeta+1} / \conf )  \rceil$
\FOR{\textbf{each}
$\nu \in \Lambda_1$
}
\STATE Add $\intero_1$ random samples distributed as $\newmes_\nu$
\ENDFOR
\STATE $\punti_1 =  \intero_1 \#(\Lambda_1)$  
\STATE $u_C^1 =  \argmin_{ v \in V_{\Lambda_1} } \| u - v \|_{\punti_1}  $
\STATE $r_1 =  u - u_C^1  $
\FOR{$\iteraz=2$ to $\iterazmax$}
\STATE $F = \textrm{BULK}(\mathcal{R}(\Lambda_{\iteraz-1}), | \langle r_{\iteraz-1},\basis_\nu \rangle_{\punti_{\iteraz-1}} |^2,\paradorfler  )$ 
\STATE $\Lambda_\iteraz = \Lambda_{\iteraz-1} \cup F$
\STATE $\intero_k= \lceil \costprop^{-1} \ln( \zeta(\paramzeta) (\#(\Lambda_k))^{\paramzeta+1} / \conf )  \rceil$
\FOR{\textbf{each}
$\nu \in \Lambda_{k-1}$
}
\STATE Add $\intero_k - \intero_{k-1}$ random samples distributed as $\newmes_\nu$
\ENDFOR
\FOR{\textbf{each}
$\nu \in \Lambda_k \setminus \Lambda_{k-1}$
}
\STATE Add $\intero_{k}$ random samples distributed as $\newmes_\nu$
\ENDFOR
\STATE $\punti_k = \intero_k \#(\Lambda_k)$  
\STATE $u_C^\iteraz =  \argmin_{ v \in V_{\Lambda_\iteraz} } \| u - v \|_{\punti_\iteraz}  $
\IF{$ \iteraz \textrm{ mod } \iteraz_{sg} =0 $}
    \STATE $\Lambda_\iteraz = \Lambda_{\iteraz-1} \cup \{ \nu \} $,  with $\nu$ being the most ancient multi-index in $\mathcal{R}(\Lambda_{\iteraz-1})\setminus F $
  \ENDIF
\STATE $r_\iteraz =  u - u_C^\iteraz  $
\ENDFOR
\end{algorithmic}
\label{algo:adaptive}
\end{algorithm}
\end{figure}

\begin{figure}[htbp]
\centering
\begin{algorithm}[H]\small
\caption{\small Fully adaptive weighted least squares}
\begin{algorithmic}
\REQUIRE $\Lambda_1=\{ (0,\ldots,0)^\top \}$, $\paradorfler$, $\iterazmax$, $\thresholdcond$, $\iteraz_{\textrm{sg}}$
\ENSURE $u_C^{\iterazmax}$
\REPEAT
\FOR{\textbf{each}
$\nu \in \Lambda_{1}$
}
\STATE Add one random sample distributed as $\newmes_\nu$
\ENDFOR
\STATE $\punti_1=\punti_1+\#(\Lambda_1)$
\UNTIL{ $\vvvert G_1 - I_1  \vvvert < \thresholdcond$  }
\STATE $u_C^1 =  \argmin_{ v \in V_{\Lambda_1} } \| u - v \|_{\punti_1}  $
\STATE $r_1 =  u - u_C^1  $
\FOR{$\iteraz=2$ to $\iterazmax$}
\STATE $F = \textrm{BULK}(\mathcal{R}(\Lambda_{\iteraz-1}), | \langle r_{\iteraz-1},\basis_\nu \rangle_{\punti_{\iteraz-1}} |^2,\paradorfler  )$ 
\STATE $\Lambda_\iteraz = \Lambda_{\iteraz-1} \cup F$
\FOR{\textbf{each}
$\nu \in \Lambda_k \setminus \Lambda_{k-1}$
}
\STATE Add $\punti_{\iteraz-1}/\#(\Lambda_{\iteraz-1})$ random samples distributed as $\newmes_\nu$
\ENDFOR
\STATE $\punti_\iteraz = \punti_{\iteraz-1} \#(\Lambda_\iteraz) /\#(\Lambda_{\iteraz-1})   $ 
\REPEAT
\FOR{\textbf{each}
$\nu \in \Lambda_{k}$
}
\STATE Add one random sample distributed as $\newmes_\nu$
\ENDFOR
\STATE $\punti_\iteraz=\punti_\iteraz+\#(\Lambda_\iteraz)$
\UNTIL{  $\vvvert G_\iteraz - I_\iteraz  \vvvert < \thresholdcond$  }
\STATE $u_C^\iteraz =  \argmin_{ v \in V_{\Lambda_\iteraz} } \| u - v \|_{\punti_\iteraz}  $
\IF{$ \iteraz \textrm{ mod } \iteraz_{sg} =0 $}
    \STATE $\Lambda_\iteraz = \Lambda_{\iteraz-1} \cup \{ \nu \} $,  with $\nu$ being the most ancient multi-index in $\mathcal{R}(\Lambda_{\iteraz-1})\setminus F $
  \ENDIF
\STATE $r_\iteraz =  u - u_C^\iteraz  $
\ENDFOR
\end{algorithmic}
\label{algo:fullyadaptive}
\end{algorithm}
\end{figure}

\end{document}